\documentclass[11pt,letterpaper]{article}
\usepackage[text={5.5in, 8.25in},centering]{geometry}
\usepackage{graphicx}
\usepackage{placeins}
\usepackage[skip=6pt]{caption}
\usepackage{epsfig,subfigure}
\usepackage{amssymb}
\usepackage{amsthm}
\usepackage{amsfonts}
\usepackage{amsmath}
\usepackage{multicol,multirow}
\usepackage{algorithm,algcompatible}
\usepackage{url}
\usepackage{todonotes}
\usepackage{xfrac}
\usepackage[sorting=ynt]{biblatex}

\usepackage{algorithm}
\usepackage{algpseudocode}

\newtheorem{theorem}{Theorem}

\newtheorem{corollary}[theorem]{Corollary}

\newtheorem{definition}[theorem]{Definition}

\usepackage{thm-restate}

\usepackage{bm} 
\providecommand{\mathbold}[1]{\bm{#1}}
\newcommand{\R}{\mathbb{R}}
\newcommand{\N}{\mathbb{N}}
\newcommand{\vct}[1]{\bm{#1}}
\newcommand{\mtx}[1]{\mathbold{#1}}
\newcommand{\inn}{\text{ in }}
\newcommand{\foralll}{\text{ for all }} 

\DeclareMathOperator{\cone}{co}
\DeclareMathOperator{\cl}{cl}
\DeclareMathOperator{\sgn}{sgn}

\newcommand{\workstation}{\bm{\mathsf{W}}}
\newcommand{\laptop}{\bm{\mathsf{L}}}
\DeclareMathOperator{\Sig}{Sig}
\DeclareMathOperator{\Poly}{Poly}
\newcommand{\cnns}[1]{\mathsf{C}_{\mathsf{NNS}}(#1)}
\newcommand{\cnnsgeo}[1]{\mathsf{C}_{\mathsf{NNS}}^{\mathsf{GEOM}}(#1)}
\newcommand{\cnnp}[1]{\mathsf{C}_{\mathsf{NNP}}(#1)}
\newcommand{\csage}[1]{\mathsf{C}_{\mathsf{SAGE}}(#1)}
\newcommand{\csagegeo}[1]{\mathsf{C}_{\mathsf{SAGE}}^{\mathsf{GEOM}}(#1)}
\newcommand{\cage}[1]{\mathsf{C}_{\mathsf{AGE}}(#1)}
\newcommand{\cpolyage}[1]{\mathsf{C}_{\mathsf{AGE}}^{\mathsf{POLY}}(#1)}
\newcommand{\cpolysage}[1]{\mathsf{C}_{\mathsf{SAGE}}^{\mathsf{POLY}}(#1)}
\newcommand{\sigreps}[1]{\mathrm{SR}(#1)}
\newcommand{\relent}[2]{D(#1, #2)}

\title{Signomial and Polynomial Optimization via Relative Entropy and Partial Dualization}

\author{Riley Murray$^\dag$, Venkat Chandrasekaran$^{\dag,\ddag}$, and Adam Wierman$^\dag$ \thanks{Email: rmurray@caltech.edu, venkatc@caltech.edu, adamw@caltech.edu \newline Acknowledgements: R.M. was
supported in part by an NSF Graduate Research Fellowship and by NSF
grant CCF-1350590, NSF grant CCF-1637598, and AFOSR grant
FA9550-16-1-0210. V.C. was supported in part by NSF grants CCF-1350590
and CCF-1637598, AFOSR grant FA9550-16-1-0210, and a Sloan Research
Fellowship. A.W. was supported in part by NSF grant CCF-1637598.} \vspace{0.25in} \\ $^\dag$ Department of Computing and Mathematical Sciences\\ $^\ddag$ Department of Electrical Engineering \\ California Institute of Technology \\ Pasadena, CA 91125}

\begin{document}

\maketitle

\begin{abstract}
	We describe a generalization of the Sums-of-AM/GM Exponential (SAGE) relaxation methodology for obtaining bounds on constrained signomial and polynomial optimization problems. Our approach leverages the fact that relative entropy based SAGE certificates conveniently and transparently blend with convex duality, in a manner that Sums-of-Squares certificates do not.
	This more general approach not only retains key properties of ordinary SAGE relaxations (e.g. sparsity preservation), but also inspires a novel perspective-based method of solution recovery.
	We illustrate the utility of our methodology with a range of examples from the global optimization literature, along with a publicly available software package.
	\newline \\
	\noindent \textbf{Keywords: global optimization, exponential cone programs, SAGE certificates, SOS certificates, signomial programming.} 
\end{abstract}

\clearpage

\section{Introduction}

A signomial is a function of the form $\vct{x} \mapsto \sum_{i=1}^m c_i \exp( \vct{\alpha}_i \cdot \vct{x})$ for real scalars $c_i$ and row vectors $\vct{\alpha}_i$ in $\R^{1 \times n}$.
Signomial optimization (often called signomial programming) concerns the minimization of a signomial, subject to signomial inequality and equality constraints.
Signomial programming is a computationally challenging problem with applications in chemical engineering \cite{Rountree1982}, aeronautics \cite{OCW2017}, circuit design \cite{jabr2007}, and communications network optimization \cite{Chiang2009}.
Signomials are sometimes thought of as generalizations of polynomials over the positive orthant; by 
a change of variables $y_i = \exp x_i$ one arrives at ``geometric form'' signomials $\vct{y} \mapsto \sum_{i=1}^m c_i \prod_{j=1}^n y_j^{\alpha_{ij}}$.
Despite this aesthetic similarity between polynomials and geometric-form signomials, we must bear in mind that signomials and polynomials have many significant differences.
Where polynomials can be generated by a countably infinite basis, signomials require an uncountably infinite basis.
Where polynomials are closed under composition, signomials are not.
Where polynomials and exponential-form signomials are defined on all of $\R^n$ -- geometric-form signomials are only defined on the positive orthant.

For many years these abstract differences between signomials and polynomials have coincided with algorithmic disparities.
Contemporary methods for signomial programming use some combination of local linearization, penalty functions, sequential geometric programming, and branch-and-bound \cite{Shen2004,WL2005,Shen2006,Qu2007,Shen2008,HSC2014,Xu2014} -- ideas which precede the advent of modern convex optimization.
By contrast, the field of polynomial optimization has been substantially influenced by semidefinite programming, specifically through Sums-of-Squares (SOS) certificates of polynomial nonnegativity \cite{shor,ParriloPhD,Lasserre2001}.
In recent work, Chandrasekaran and Shah proposed the Sums-of-AM/GM Exponential or “SAGE” certificates of signomial nonnegativity, which provided a new convex relaxation framework for signomial programs akin to SOS methods for polynomial optimization.  \cite{SAGE1}.
Where SOS certificates make use of semidefinite programming, SAGE certificates use the convex \textit{relative entropy function}.
The authors of the present article further demonstrated that a natural modification to SAGE certificates leads to a tractable relative entropy representable sufficient condition for global polynomial nonnegativity \cite{SAGE2}.

This article is concerned with how proof systems for function nonnegativity can be used in the service of constrained optimization.
The basic idea here is simple: for a function $f$, a set $X$, and a real number $\gamma$, we have $\inf\{f(\vct{x}) \,:\, \vct{x} \in X\} \geq \gamma$ if and only if $f - \gamma$ is nonnegative over $X$.
The trouble is that to leverage this fact, we require ways to extend certificates for \textit{global} nonnegativity (such as SOS or SAGE certificates) to prove nonnegativity over $X \subsetneq \R^n$.
For the polynomial case one usually performs this extension by appealing to representation theorems from real algebraic geometry.
In the absence of such representation theorems, one typically relies on a dual problem obtained from the minimax inequality.

The primary contribution of this article is to show how SAGE certificates -- by virtue of their roots in convex duality -- provide a simple and powerful alternative method for describing functions which are nonnegative over proper subsets of $\R^n$.
Our method can be used both independently from and in conjunction with the minimax inequality.
The space of possibilities with our method is large, and it is far from obvious as to which variations of this methodology are most useful for given problem structures.
To facilitate research in this regard, we provide a user-friendly software package which implements all functionality described in this article.
We provide detailed worked examples in several places alongside conceptual development.
A dedicated section on computational experiments is provided, and several avenues of possible future research are outlined in a discussion section.

\subsection{Article outline and our contributions}

This article makes both mathematical and methodological contributions to signomial and polynomial optimization.
Section \ref{sec:background} speaks to key questions which help place our work in a broader context.
These questions include
(1) What are the sources of error in nonnegativity-based relaxations of constrained optimization problems, and how are they usually mitigated?
(2) How exactly are the original SAGE cones formulated?
(3) How can we understand partial dualization in the context of existing nonnegativity and moment relaxations?

Once these questions are answered, we introduce the concept of conditional SAGE certificates for signomial nonnegativity (Section \ref{sec:condsage_sigs}).
We prove a representation result for the cone of these nonnegativity certificates (Theorem \ref{thm:sigsage_represent_age}), and develop a solution recovery algorithm by investigating the dual cone (Algorithm \ref{alg:sp_solrec}).
Section \ref{sec:condsage_sigs:ref_hier} describes two ``hierarchies'' of SAGE-based convex relaxations for signomial programs: one which uses the minimax inequality, and one which is minimax-free.
The authors know of no analog to the minimax-free hierarchy in the polynomial optimization literature, and believe the underlying idea of the minimax-free hierarchy is of independent theoretical interest.

Section \ref{sec:condsage_polys} extends the idea of conditional SAGE certificates to polynomials.
We discuss basic properties of the conditional SAGE polynomial cones before proving representation results (Theorems \ref{thm:reduce_polysage_to_sigsage_2} and \ref{thm:reduce_polysage_to_sigsage}) which provide the basis for tractable relaxations of constrained polynomial optimization problems.
Section \ref{sec:condsage_polys:solrec} provides simple descriptions for dual conditional SAGE polynomial cones, and develops an efficient solution recovery algorithm based on these descriptions (Algorithm \ref{alg:pop_solrec}).
Section \ref{sec:condsage_polys:refhier} proposes reference hierarchies for polynomial optimization with SAGE certificates.
Our minimax-free hierarchy has an interesting structure which reflects a link between SAGE signomials and SAGE polynomials, by way of the ``signomial representatives'' from \cite{SAGE2}.

Section \ref{sec:experiments} reports the effectiveness of our methodology on fifty-one problems appearing in the literature (sourced from \cite{Yan1976,RM1978,RN2008,HSC2014,Xu2014,LTY2017,WLT2018,VLSE}), as well as randomly generated problems.
A central component of our experiments is a desire to facilitate research both into theory underlying conditional SAGE relaxations, and the practice of using these relaxations in engineering design optimization.
Towards this end, we provide the  ``\texttt{sageopt}'' Python package.\footnote{\url{https://rileyjmurray.github.io/sageopt/}} \texttt{Sageopt} is a documented, tested, and convenient platform for constructing and solving SAGE relaxations, as well as analyzing the results thereof.
We used \texttt{sageopt} for all experiments in this article.

\subsection{Notation and preliminary definitions}

Vectors and matrices always appear in boldface.
The $i^{\text{th}}$ entry of a vector $\vct{v}$ is $v_i$, and the vector formed by deleting the $i^{\text{th}}$ entry of $\vct{v}$ is $\vct{v}_{\setminus i}$.
A matrix $\mtx{A}$ is built by stacking rows $\vct{a}_i \in \R^{1 \times n}$, and $\mtx{A}_{\setminus i}$ is the submatrix formed by deleting the $i^{\text{th}}$ row of $\mtx{A}$.
All logarithms in this article are base-$e$. 
Elementary functions from $\R$ to $\R$ are extended first to vectors in an elementwise fashion, and subsequently to sets in a pointwise fashion.
For a convex cone $K \subset \R^r$, the dual cone is $K^\dagger \doteq \{ \vct{y} \,:\, \vct{y}^\intercal \vct{x} \geq 0 \foralll \vct{x} \inn \R^r \}$.
For $A, B \subset \R^n$, $A \subset B$ and $A \subsetneq B$ denote non-strict and strict inclusion respectively.
The operator ``$\cl$'' computes set-closure with respect to the standard topology.

For an $m \times n$ matrix $\mtx{\alpha}$ and a vector $\vct{c}$ in $\R^m$, we write $f = \Sig(\mtx{\alpha},\vct{c})$ to mean that $f$ takes values $f(\vct{x}) = \sum_{i=1}^m c_i \exp(\vct{\alpha}_i \cdot \vct{x})$.
When $\mtx{\alpha}$ is a matrix of nonnegative integers, we write $f = \Poly(\mtx{\alpha},\vct{c})$ to mean that $\vct{c}$ is the coefficient vector of $f$ with respect to the monomial basis $\vct{x} \mapsto \vct{x}^{\vct{\alpha}_i} \doteq \prod_{j=1}^n x_j^{\alpha_{ij}}$.
Given a matrix $\mtx{\alpha}$ and a set $X \subset \R^n$, one has the nonnegativity cones
\[
\cnns{\mtx{\alpha},X} \doteq \{ \vct{c} \,:\, \Sig(\mtx{\alpha},\vct{c})(\vct{x}) \geq 0 \foralll \vct{x} \inn X \}
\]
and
\[
\cnnp{\mtx{\alpha},X} \doteq \{ \vct{c} \,:\, \Poly(\mtx{\alpha},\vct{c})(\vct{x}) \geq 0 \foralll \vct{x} \inn X \}.
\]
We write $\cnns{\mtx{\alpha}}$ and $\cnnp{\mtx{\alpha}}$ in reference to the above cones when $X = \R^n$.
Except in special cases on $\mtx{\alpha}$, it is computationally intractable to check membership in either $\cnns{\mtx{\alpha}}$ or $\cnnp{\mtx{\alpha}}$ \cite{Murty1987}.
The inner-approximations of nonnegativity cones developed in this article make use of the relative entropy function;
this is the convex function ``$D$'' with domain $\R^m_+ \times \R^m_+$ taking values
\[
 \relent{\vct{u}}{\vct{v}} = \sum_{i=1}^m u_i \log(u_i / v_i).
\]

This article includes computational experiments with SAGE certificates and states solver runtimes for many of these examples.
All of these examples rely on the MOSEK solver \cite{mosek}.
We use two different machines to provide a sense of when it may be practical to solve a SAGE relaxation with given computational resources.
Machine $\workstation$ is an HP Z820 workstation, with two 8-core 2.6GHz Intel Xeon E5-2670 processors and 256GB 1600MHz DDR3 RAM.
Machine $\laptop$ is a 2013 MacBook Pro, with a dual-core 2.4GHz Intel Core i5 processor and 8GB 1600MHz DDR3 RAM.

\section{Background}\label{sec:background}

In this article we study constrained nonconvex optimization problems of the form
\begin{equation}
    (f,g,\phi)^\star_{X} = \inf\{ f(\vct{x}) : \vct{x} \inn X \subset \R^n, ~ g(\vct{x}) \geq \vct{0}, ~\phi(\vct{x}) = \vct{0} \} \label{eq:background:genericConstrainedMin}
\end{equation}
where $f$ is a function from $\R^n$ to $\R$, $g$ maps $\R^n$ to $\R^{k_1}$, and $\phi$ maps $\R^n$ to $\R^{k_2}$.
Our primary goal is to produce lower bounds $(f,g,\phi)^{\mathrm{lb}}_X \leq (f,g,\phi)^\star_X$.
In the event that $(f,g,\phi)^{\mathrm{lb}}_X = (f,g,\phi)^\star_X$, we are also interested in recovering optimal solutions to \eqref{eq:background:genericConstrainedMin}.
For ease of exposition, this section focuses on problems of the form \eqref{eq:background:genericConstrainedMin} with only inequality constraints-- i.e. the problem of bounding
\begin{equation}
    (f,g)^\star_{X} = \inf\{ f(\vct{x}) : \vct{x} \inn X \subset \R^n, ~ g(\vct{x}) \geq \vct{0} \}. \tag{1.1} \label{eq:background:ineqConstrainedMin}
\end{equation}

In Section \ref{sec:background:duals_in_nonconvex_opt} we review the Lagrange dual relaxation of the above problem, both in minimax form and as a nonnegativity problem.
Section \ref{sec:background:sage_and_sos} provides the minimum background on SAGE and SOS nonnegativity certificates needed develop the contributions of this article.
In Section \ref{sec:background:strengthen} we review standard techniques for strengthening nonnegativity-based relaxations of problems such as \eqref{eq:background:ineqConstrainedMin}; this includes the use of redundant constraints, nonconstant Lagrange multipliers, and strengthening nonnegativity certificates via modulation.
Section \ref{sec:background:partial_dualization} concludes with discussion on partial dualization.
Until Section \ref{sec:background:partial_dualization}, the set $X$ appearing in Problem \ref{eq:background:ineqConstrainedMin} shall be the whole of $\R^n$.

\subsection{Dual problems in nonconvex optimization}\label{sec:background:duals_in_nonconvex_opt}

The simplest way to lower bound $(f,g)^\star_{\R^n}$ is via the Lagrange dual.
For each coordinate function $g_i$ of $g$, we introduce a dual variable $\lambda_i \geq 0$ and consider the Lagrangian $\mathcal{L}(\vct{x},\vct{\lambda}) = f(\vct{x}) - \vct{\lambda}^\intercal g(\vct{x})$.
The \textit{Lagrange dual problem} is to compute
\begin{equation*}
(f,g)^{\mathrm{L}}_{\R^n} = \sup_{\vct{\lambda} \geq \vct{0}} \inf_{\vct{x} \in \R^n} \mathcal{L}(\vct{x},\vct{\lambda}).
\end{equation*}
By the minimax inequality, we can be certain that $(f,g)^{\mathrm{L}}_{\R^n} \leq (f,g)^\star_{\R^n}$.

There are many situations when the Lagrange dual problem is intractable.
For signomial and polynomial optimization, one usually needs to compute yet another lower bound $(f,g)_{\R^n}^{\mathrm{d}} \leq (f,g)_{\R^n}^{\mathrm{L}}$.
Contemporary approaches for computing such bounds begin by introducing a parameterized function $\psi(\gamma,\vct{\lambda})$ which takes values $\psi(\gamma,\vct{\lambda})(\vct{x}) =  \mathcal{L}(\vct{x},\vct{\lambda}) - \gamma$.
One reformulates the dual problem as
\[
(f,g)^{\mathrm{L}}_{\R^n} = \sup\{  \gamma ~: ~\vct{\lambda} \geq \vct{0},~ \gamma \inn \R, ~ \psi(\gamma,\vct{\lambda})(\vct{x}) \geq 0 \foralll \vct{x} \inn \R^n \},
\]
and the constraint that ``$\psi(\gamma,\vct{\lambda})$ defines a nonnegative function'' is then tightened to ``$\psi(\gamma,\vct{\lambda})$ satisfies a particular sufficient condition for nonnegativity.''
The expectation is that the sufficient condition can be expressed by tractable convex constraints on variables $\gamma$ and $\vct{\lambda}$.
For example, SOS certificates for polynomial nonnegativity can be expressed via linear matrix inequalities, and SAGE certificates for signomial and polynomial nonnegativity can be expressed with the relative entropy function.

\subsection{SAGE and SOS nonnegativity certificates}\label{sec:background:sage_and_sos}

In the development of the SAGE inner approximation for $\cnns{\mtx{\alpha}}$, Chandrasekaran and Shah considered the structure where the coefficient vector $\vct{c}$ contained at most one negative entry $c_k$; if such a function was globally nonnegative, they called it an \textit{AM/GM Exponential}, or an \textit{AGE function} \cite{SAGE1}.
One thus defines the \textit{$k^{\text{th}}$ AGE cone}
\begin{equation}
\cage{\mtx{\alpha},k} = \{  \vct{c} : \vct{c}_{\setminus k} \geq \vct{0} \text{ and } \vct{c} \text{ belongs to } \cnns{\mtx{\alpha}} \}. \label{eq:background:defAGE} \nonumber
\end{equation}
A key contribution of \cite{SAGE1} was the use of convex duality to derive an efficient description of the AGE cones.
The outcome of this derivation is that a vector $\vct{c}$ belongs to $\cage{\mtx{\alpha},k}$ iff $\vct{c}_{\setminus k} \geq \vct{0}$ and
\begin{equation}
\text{some} \quad \vct{\nu} \inn \R^{m-1}_+ \quad \text{has} \quad [\mtx{\alpha}_{\setminus i} - \vct{1}\vct{\alpha}_i ]^\intercal \vct{\nu} = \vct{0} \quad \text{and}  \quad  \relent{\vct{\nu}}{\vct{c}_{\setminus k}} - \vct{\nu}^\intercal \vct{1}  \leq c_k. \label{eq:background:ageRelEnt}
\end{equation}
The system of constraints given by \eqref{eq:background:ageRelEnt} is crucially jointly convex in $\vct{c}$ and the auxiliary variable $\vct{\nu}$.
The set defined by the \textit{sum} of all AGE cones
\begin{equation}
\csage{\mtx{\alpha}} \doteq \left\{ \vct{c} : \text{ there exist } \vct{c}^{(k)} \inn \cage{\mtx{\alpha},k} \text{ satisfying } \vct{c} = \sum_{k=1}^m \vct{c}^{(k)} \right\}\label{eq:background:csagedef}
\end{equation}
is therefore efficiently representable.

The SAGE cone as defined above applies to signomials, but a similar construction exists for certifying global nonnegativity of polynomials \cite{SAGE2}.
Formally, we say that $f = \Poly(\mtx{\alpha},\vct{c})$ is an \textit{AGE polynomial} if it is nonnegative over $\R^n$, and if $f(\vct{x})$ contains at most one term $c_i \vct{x}^{\vct{\alpha}_i}$ that is not a monomial square.
In conic form this writes as
\begin{align}
\cpolyage{\mtx{\alpha},k} = \{ \vct{c} :&~\Poly(\mtx{\alpha},\vct{c})(\vct{x}) \geq 0 \foralll \vct{x} \inn \R^n \text{, and} \nonumber \\
	&~ \vct{c}_{\setminus k} \geq \vct{0} , ~c_i = 0 \foralll i \neq k \text{ with } \vct{\alpha}_i \not\in 2\mathbb{N}^{ 1 \times n} \}, \label{eq:def_poly_age_ordinary}
\end{align}
and such AGE cones naturally give rise to
\begin{equation}
\cpolysage{\mtx{\alpha}} \doteq \sum_{k=1}^m \cpolyage{\mtx{\alpha},k} \subset \cnnp{\mtx{\alpha}}. \label{eq:def:c_poly_sage}
\end{equation}

The SAGE polynomial cone can also be described by an appropriate reduction to the SAGE signomial cone.
For a nonnegative $m \times n$ integer matrix $\mtx{\alpha}$ and a vector $\vct{c}$ in $\R^m$, we define the set of \textit{signomial representative coefficient vectors} as
\begin{align}
\sigreps{\mtx{\alpha},\vct{c}} = \{ \vct{\hat{c}} &:\, \hat{c}_i = c_i \text{ whenever } \vct{\alpha}_i  \text{ is in } 2\mathbb{N}^{1 \times n}, \text{ and } \nonumber  \\
&~~~ \hat{c}_i \leq -|c_i| \text{ whenever } \vct{\alpha}_i \text{ is not in } 2\mathbb{N}^{1 \times n} \}. \nonumber
\end{align}
The name ``signomial representative'' derives from the fact that if $\vct{\hat{c}}$ belongs to $\sigreps{\mtx{\alpha},\vct{c}}$, then nonnegativity of the signomial $ \Sig(\mtx{\alpha},\vct{\hat{c}})$ would evidently imply nonnegativity of the polynomial $\Poly(\mtx{\alpha},\vct{c})$ (see Section 5.1 of \cite{SAGE2}).
The set $\sigreps{\mtx{\alpha},\vct{c}}$ is useful because a constraint of the form ``$\vct{\hat{c}}$ belongs to $\sigreps{\mtx{\alpha},\vct{c}}$'' is jointly convex in $\vct{\hat{c}}$ and $\vct{c}$.
Lemma 19 of \cite{SAGE2} proves that the SAGE polynomial cone defined by Equation \ref{eq:def:c_poly_sage} is equivalently given by
\begin{align*}
\cpolysage{\mtx{\alpha}} = \{ \vct{c} ~:&~ \sigreps{\mtx{\alpha},\vct{c}} \cap \csage{\mtx{\alpha}} \text{ is nonempty } \}.
\end{align*}
The generalization of SAGE polynomials considered in this article benefits from both the Sum-of-AGE-function and signomial-representative viewpoints of ordinary SAGE polynomials.

Lastly we consider \textit{Sums-of-Squares} (SOS) polynomials.
A polynomial $f$ is said to be SOS if it can be written in the form $f = \sum_{i=1}^m f_i^2$ for appropriate polynomials $f_i$.
In the context of polynomial optimization, one usually parameterizes the SOS cone by a number of variables $n$ and a maximum degree $2d$; this cone can be represented as
\[
\mathrm{SOS}(n, 2d) = \{ p \,:\, p(\vct{x}) = L^n_d(\vct{x})^\intercal\mtx{M}L^n_d(\vct{x}),~ \mtx{M} \succeq \mtx{0} \}
\]
where $L^n_d : \mathbb{R}^n  \to \R^{ {n + d \choose d} }$ is the map from a vector $\vct{x}$ to the vector of all monomials of degree at-most-$d$ evaluated at $\vct{x}$.
The connection between SOS-representability and semidefinite programming was first observed by Shor \cite{shor}, and was subsequently developed by Parrilo \cite{ParriloPhD} and Lasserre \cite{Lasserre2001}.

\subsection{Strengthening dual bounds in nonnegativity relaxations}\label{sec:background:strengthen}

A common method for strengthening dual problems is to introduce redundant constraints to the primal problem, particularly by taking products of existing constraint functions.
As an example of this principle in action, consider the toy polynomial optimization problem
\[
\inf\{\, -x^2 \,:\, -1 \leq x \leq 1 \, \} = -1.
\]
One may verify that $(f,g)^{\mathrm{L}}_{\R} = -\infty$, but by adding the single redundant constraint $(1-x)(1+x) \geq 0$, we can certify a dual bound $-1 \leq (f,g)_{\R}^\star$.

A more subtle method is to reconsider what is meant by ``dual variables.''
For the Lagrange dual problem we use scalars $\lambda_i \geq 0$, however it would be just as valid to have $\lambda_i$ be a \textit{function}, provided that it was nonnegative over $\R^n$. 
Such a method is well-suited to our nonnegativity-based relaxations of the dual problem.
The following toy signomial program illustrates the utility of this approach
\begin{equation*}
\inf\{ -\exp(2x) ~:~ 1 \leq \exp(x) \leq 2\} = -4. \label{eq:background:toySP}
\end{equation*}
Again the Lagrange dual problem returns a bound of $-\infty$, but by considering $\lambda_i$ of the form $\lambda_i(x) = \eta_i \exp(x)$ with $\eta_i \geq 0$, the resulting dual bound is $-4 \leq (f,g)_{\R}^\star$.

Our third method for strengthening dual bounds only becomes relevant when working with strict inner-approximations of nonnegativity cones.
For two functions $w, f$ with $w$ positive definite, it is clear that $f$ is nonnegative if and only if the product $w \cdot f$ is nonnegative.
The \textit{method of modulation} is to choose a generic positive-definite function $w$ so that if $f$ fails a particular test for nonnegativity (say, being SOS, or being SAGE), there is still a chance that the product $w \cdot f$ passes a test for nonnegativity.
Indeed, modulation is a crucial tool for computing successive bounds for unconstrained problems
\[
f^\star_{\R^n} \doteq \inf\{ f(\vct{x}) : \vct{x} \inn \R^n\} = \sup\{ \gamma : f(\vct{x}) - \gamma \geq 0 \foralll \vct{x} \inn \R^n \}.
\]
Suppose for example that $f$ is a signomial over exponents $\mtx{\alpha}$; then for $w = \Sig(\mtx{\alpha},\vct{1})$ we can compute a non-decreasing sequence of lower bounds
\[
f^{(\ell)}_{\R^n} = \sup\{ \gamma : \gamma \inn \R,~ w^\ell (f - \gamma) \text{ is SAGE} \} \leq f^\star_{\R^n}.
\]
Under appropriate conditions on $\mtx{\alpha}$ (c.f. \cite{SAGE1}), these lower bounds converge to $f^\star_{\R^n}$ as $\ell$ goes to infinity.
From an implementation perspective, the constraint that ``$\psi(\gamma) \doteq w^\ell(f - \gamma)$ is SAGE'' is tractable because the coefficient vector of $\psi(\gamma)$ is an affine function of $\gamma$.

Modulation can similarly be applied to constrained optimization.
Suppose that $\mathcal{L}(\vct{x},\vct{\lambda})$ is the Lagrangian for Problem \ref{eq:background:ineqConstrainedMin}, and refer to the function $\vct{x} \mapsto \mathcal{L}(\vct{x},\vct{\lambda})$ as $\mathcal{L}(\vct{\lambda})$.
Then rather than requiring that ``$\mathcal{L}(\vct{\lambda}) - \gamma$ is SAGE'', one could require that ``$\psi(\gamma,\vct{\lambda}) \doteq w^\ell( \mathcal{L}(\vct{\lambda}) - \gamma ) \text{ is SAGE}$.''
This would increase the size of the feasible set for variables $\gamma$ and $\vct{\lambda}$, and remain tractable due to the affine dependence of $\psi(\gamma, \vct{\lambda})$ on $\gamma$ and $\vct{\lambda}$.
Such modulation leads to a non-decreasing sequence of bounds which converge to $(f,g)_{\R^n}^{\mathrm{L}}$ under suitable conditions.

\subsection{Partial dualization}\label{sec:background:partial_dualization}

A \textit{partial dual problem} is what results when the set ``$X$'' in Problem \ref{eq:background:ineqConstrainedMin} is a proper subset of $\R^n$.
In this case the natural generalization of the Lagrange dual is
\begin{equation}
(f, g)_{X}^{\mathrm{d}} \doteq \sup\{\, \gamma \,:\, \vct{\lambda} \geq \vct{0},~\gamma \inn \R,~ \mathcal{L}(\vct{x},\vct{\lambda})- \gamma \geq 0 \foralll \vct{x} \inn X \}. \label{eq:partialDualDef}
\end{equation}
The technique of \textit{partial dualization} refers to the deliberate choice to restrict the Lagrangian to $X = \{ \vct{x} : g_i(\vct{x}) \geq 0 \foralll i \inn \mathcal{I}\}$ for some $\mathcal{I} \subset [k]$, even when the constraints $\{g_i\}_{i \in \mathcal{I}}$ are of a functional form that is permitted in the Lagrangian.
Note that in the extreme case with $X = \{ \vct{x} : g(\vct{x}) \geq \vct{0}\}$, we are certain to have $(f, 0)^{\mathrm{d}}_{X} = (f,g)^\star_{\R^n}$ --
in this way, partial dualization provides a mechanism to completely eliminate duality gaps.

Before getting into how SAGE certificates integrate with partial dualization, it is worth considering a simple example which combines partial dualization and nonnegativity certificates.
Suppose we want to minimize a univariate polynomial $f$ over an interval $[a, b]$, subject to a single polynomial equality constraint $g(x) = 0$.
In this case we could form a Lagrangian $\mathcal{L}(x,\mu) = f(x) - \mu g(x)$ with $\mu \in \R$, and find the largest constant $\gamma$ so that $\mathcal{L}(x,\mu) - \gamma$ was nonnegative over $x \in [a, b]$.
A well-known result in real algebraic geometry is that a degree-$d$ polynomial ``$p$'' is nonnegative over an interval $[a, b]$ if and only if $p$ can be written as 
$p(x) = s(x)^2 + h_{[a,b]}(x) t(x)^2$,
where $h_{[a, b]}(x) = (b-x)(x-a)$, and $s$, $t$ are polynomials of degree at most $d$ and $d - 1$ respectively \cite{Powers2000}.
Therefore the partial dual problem
\[
(f,g)_{[a,b]}^\mathrm{d} = \sup\{ \gamma \,:\, \gamma, \mu \in \R, ~ f(x) - \mu g(x) - \gamma  \geq 0 \foralll x \inn [a, b] \}
\]
can be framed as an SOS relaxation
\begin{align*}
(f,g)_{[a,b]}^{\mathrm{d}} = \sup\{ \gamma \,:\,& f - \mu g - \gamma = s + h_{[a, b]} t \\
                                                &s \in \mathrm{SOS}(1,2d), ~ t \in \mathrm{SOS}(1, 2(d-1)) \}.
\end{align*}

Our last key concept is how partial dualization manifests in the dual of the dual.
To develop this idea, consider $f = \Poly(\mtx{\alpha},\vct{c})$ with $\vct{\alpha}_1 = \vct{0}$, along with a set $X \subset \R^n$.
The problem of computing $f_X^\star \doteq \inf_{\vct{x} \in X} f(\vct{x})$ has the following convex formulation
\[
    f_X^\star = \sup\{ \, \gamma \,:\, \vct{c} - \gamma (1,0,\ldots,0) ~\in~ \cnnp{\mtx{\alpha},X} \}.
\]
Of course, the above problem is intractable unless $\mtx{\alpha}$ and $X$ satisfy very special conditions.
In spite of the possible intractability, we can still compute the dual problem by applying standard rules of conic duality.
The result of this process is
\[
f_X^\star = \inf\{ \vct{c}^\intercal \vct{v} ~:~ \vct{v} \in \cnnp{\mtx{\alpha},X}^\dagger,~ v_1 = 1 \}
\]
where  $\cnnp{\mtx{\alpha},X}^\dagger$ is the dual cone to $\cnnp{\mtx{\alpha},X}$.
This second problem is what we mean by ``the dual of the dual.''
It appears prominently in the literature on polynomial optimization, where it is usually referred to as a \textit{moment relaxation} \cite{Lasserre2001}.
The term ``moment relaxation'' derives from the fact that $\cnnp{\mtx{\alpha},\R^n}^\dagger$ is the smallest closed convex cone containing the vectors 
\[
 (\R^n)^{\mtx{\alpha}} \doteq \{ (\vct{x}^{\vct{\alpha}_1},\ldots,\vct{x}^{\vct{\alpha}_m}) \,:\, \vct{x} \in \R^n \},
\]
and by thinking of a convex hull as computing expectations $\mathbb{E}_{\vct{x} \sim F}[(\vct{x}^{\vct{\alpha}_1},\ldots,\vct{x}^{\vct{\alpha}_m})]$, where $F$ is a probability measure over $\R^n$.
One can similarly understand the dual of a nonnegativty-based partial-dual problem in terms of probability and moment relaxations.
When $X$ as a proper subset of $\R^n$, the convex hull of $X^{\mtx{\alpha}}$ can be framed as the set of all vector-valued expectations $\mathbb{E}_{\vct{x} \sim F}[(\vct{x}^{\vct{\alpha}_1},\ldots,\vct{x}^{\vct{\alpha}_m})]$, where $F$ is a probability measure over $X$.
In this way, the ``dual'' of partial dualization can be understood in terms of \textit{conditional moments}.

\section{Conditional SAGE certificates for signomials}\label{sec:condsage_sigs}

In this section we show how SAGE certificates for signomial nonnegativity can fully leverage partial dualization, in the sense that any efficiently representable convex set $X$ gives rise to a parameterized and efficiently representable ``$X$-SAGE'' nonnegativity cone.
The efficient representation of the $X$-SAGE cones (which we often call ``conditional SAGE cones'') leads to a practical, principled approach for solving and approximating a range of nonconvex signomial optimization problems.
In this regard the most common sets $X$ are of the form $\{ \vct{x} : g(\vct{x}) \leq \vct{1} \}$ for signomials $g_i$ with all nonnegative coefficients.
An algorithm for solution recovery, and two worked examples are provided.

\subsection{The conditional SAGE signomial cones}\label{sec:condsage_sigs:main}

\begin{definition}[Conditional AGE signomial cones]\label{def:sig_cond_age}
For a matrix $\mtx{\alpha}$ in $\R^{m \times n}$, a subset $X$ of $\R^n$, and an index $k$ in $[m]$, the $k^{\text{th}}$ \text{AGE cone} with respect to $\mtx{\alpha}, X$ is
\[
\cage{\mtx{\alpha},k,X} = \{ \vct{c} \in \R^m \,:\, \vct{c}_{\setminus k} \geq \vct{0} \text{ and } \Sig(\mtx{\alpha},\vct{c})(\vct{x}) \geq 0 \foralll \vct{x} \inn X \}.
\]
\end{definition}

\begin{definition}[$X$-SAGE signomials]\label{def:sig_Xusage}
If the vector $\vct{c}$ belongs to
\[
\csage{\mtx{\alpha},X} \doteq \sum_{k=1}^m \cage{\mtx{\alpha},k,X}
\]
then $f = \Sig(\mtx{\alpha},\vct{c})$ is an $X$-SAGE signomial.
\end{definition}

Conditional SAGE cones are order-reversing with respect to the second argument.
That is, if $X_2 \subset X_1 \subset \R^n$, then $\csage{\mtx{\alpha},X_1} \subset \csage{\mtx{\alpha},X_2}$ for all $\mtx{\alpha}$ in $\R^{m \times n}$.
Note that $\csage{\mtx{\alpha},X}$ is defined for arbitrary $X \subset \R^n$, including nonconvex sets, and convex sets which admit no efficient description. Moreover, as a mathematical object, $\csage{\mtx{\alpha},X}$ does not depend on the representation of $X$.

In practice we need conditions on $X$ in order to optimize over $\csage{\mtx{\alpha},X}$.
But before we get to those, it is worth mentioning some abstract results concerning optimization.
For $f = \Sig(\mtx{\alpha}, \vct{c})$ with $\vct{\alpha}_1 = \vct{0}$, define
\[
f^{\mathrm{SAGE}}_{X} \doteq \sup\{ \, \gamma \,:\, \gamma \inn \R,~ \vct{c} - \gamma (1,0,\ldots,0)  \inn \csage{\mtx{\alpha},X} \}
\]
so that $f^{\mathrm{SAGE}}_{X} \leq f_X^\star \doteq \inf\{ f(\vct{x}) : \vct{x} \inn X \}$.

\begin{theorem}\label{thm:sigsage_exact_for_posynomials}
If $\vct{c} \geq \vct{0}$, then $f = \Sig(\mtx{\alpha},\vct{c})$ has $f^{\mathrm{SAGE}}_{X} = f_X^\star$ for all $X \subset \R^n$.
\end{theorem}
\begin{proof}
Let $\vct{e}_1 = (1,0,\ldots,0)$.
The signomial $\tilde{f} = \Sig(\mtx{\alpha}, \vct{c} - f_X^\star \vct{e}_1)$ is nonnegative over $X$, and its coefficient vector $\vct{c} - f_X^\star \vct{e}_1$ contains at most one negative entry. This implies that $\tilde{f}$ is $X$-AGE, and hence $X$-SAGE.
\end{proof}

\begin{theorem}\label{thm:sigsage_bounded_error}
If $X$ is bounded, then $f^{\mathrm{SAGE}}_{X} > -\infty$ for every signomial $f$.
\end{theorem}
\begin{proof}
If $X$ is empty then the result follows by verifying that $\csage{\mtx{\alpha},X} = \R^m$.
Consider the case when $X$ is nonempty.
In this situation it suffices to prove the result for all $f$ of the form $f(\vct{x}) = c\exp(\vct{a} \cdot \vct{x})$ where $c \neq 0$ and $\vct{a}$ belongs to $\R^{1 \times n}$.
Fixing such $c$, $\vct{a}$, the boundedness of $X$ implies the existence of $L \neq 0$ with $\tilde{f}(\vct{x}) = c \exp(\vct{a}^\intercal \vct{x}) + L$ nonnegative over $\vct{x}$ in $X$ and $c L < 0$.
Since $\tilde{f}$ is nonnegative over $X$ and contains exactly one negative coefficient, we have that $f^{\mathrm{SAGE}}_{X} \geq -L$.
\end{proof}

\begin{corollary}[See \cite{SAGE2}]\label{cor:sigsage_sparsity}
Let $X \subset \R^n$ be arbitrary. If $\vct{c}$ is a vector in $\csage{\mtx{\alpha},X}$ with nonempty $\mathcal{N} = \{ i : c_i < 0 \}$, then there exist vectors $\{\vct{c}^{(i)} \}_{i \in \mathcal{N}}$ satisfying
\[
 \vct{c}^{(i)} \in \cage{\mtx{\alpha},i,X} \qquad \vct{c} = \sum_{i \in \mathcal{N}} \vct{c}^{(i)} \quad \text{ and } \quad c^{(i)}_j = 0 \foralll j \neq i \inn \mathcal{N}.
\]
\end{corollary}
\begin{proof}
This is simply the statement of Theorem 2 from \cite{SAGE2}, which was proven for ordinary SAGE cones, i.e. with $X = \R^n$.
The entire proof of that theorem (including Lemmas 6 and 7 of \cite{SAGE2}) extends to conditional SAGE cones simply by replacing references to ``$\cage{\mtx{\alpha},i}$'' and ``$\csage{\mtx{\alpha}}$'' with ``$\cage{\mtx{\alpha},i,X}$'' and ``$\csage{\mtx{\alpha},X}$'' respectively.
\end{proof}

In order to reliably optimize over $\csage{\mtx{\alpha},X}$, we need $X$ to be a tractable convex set.
This is essentially the only requirement on $X$, as is shown by the following theorem.

\begin{theorem}\label{thm:sigsage_represent_age}
For a matrix $\mtx{\alpha}$ in $\R^{m \times n}$, an index $i$ in $[m]$, and a convex set $X \subset \R^n$ with support function $\sigma_X(\vct{\lambda}) \doteq \sup_{\vct{x} \in X} \vct{\lambda}^\intercal \vct{x} $, we have
\begin{align*}
	\cage{\mtx{\alpha},i,X} = \{ \vct{c} :&~  \vct{\nu} \text{ in } \mathbb{R}^{m-1}, \vct{c} \text{ in } \mathbb{R}^m, \vct{\lambda} \text{ in } \R^n \text{ satisfy} \nonumber \\
	&~ \sigma_X(\vct{\lambda}) + \relent{\vct{\nu}}{\vct{c}_{\setminus i}} - \vct{\nu}^\intercal \vct{1} \leq c_i, \nonumber \\
	&~[\mtx{\alpha}_{\setminus i} - \vct{1}\vct{\alpha}_i]^\intercal\vct{\nu} +  \vct{\lambda} = \vct{0},\text{ and } \vct{c}_{\setminus i} \geq \vct{0} \}.
\end{align*}
\end{theorem}
\begin{proof}
Let $\delta_X$ denote the indicator function of $X$, taking values
\[
    \delta_X(\vct{x}) = \begin{cases} 0 &\text{ if } \vct{x} \text{ belongs to } X \\
                                      +\infty &\text{ if } \text{otherwise }
                        \end{cases}.
\]
A vector $\vct{c}$ with $\vct{c}_{\setminus i} \geq \vct{0}$ belongs to $\cage{\mtx{\alpha},i,X}$ if and only if
\begin{equation}
 p^\star = \inf\{ \delta_X(\vct{x}) + \textstyle\sum_{i=1}^\ell \tilde{c}_i \exp t_i \,:\, \vct{x} \in \R^n, ~ \vct{t} \in \R^\ell, ~ \vct{t} = \mtx{W}\vct{x} \} \geq -L \label{eq:primal_age_conic}
\end{equation}
for $\ell = m - 1$, $\mtx{W} = [\mtx{\alpha}_{\setminus i} - \vct{1}\vct{\alpha}_i] \in \R^{\ell \times n}$, $\vct{\tilde{c}} = \vct{c}_{\setminus i} \in \R^{\ell}$, and $L = c_i$.

The dual to the above optimization problem is easily calculated by applying Fenchel duality (c.f. \cite{Borwein2006}); the result of this process is
\begin{equation}
d^\star = \sup\{ - \sigma_X(\vct{\lambda}) - \relent{\vct{\nu}}{\vct{\tilde{c}}} + \vct{\nu}^\intercal\vct{1} \,:\, \vct{\lambda} \in \R^n, \vct{\nu} \in \R^{m-1}, ~ \mtx{W}^\intercal \vct{\nu} + \vct{\lambda} = \vct{0} \}. \label{eq:dual_age_conic_proof}
\end{equation}
When $X$ is nonempty, one may verify that the hypothesis of Corollary 3.3.11 of \cite{Borwein2006} (concerning strong duality) hold for the primal-dual pair \eqref{eq:primal_age_conic}-\eqref{eq:dual_age_conic_proof}. In particular, $p^\star \geq -L$ holds if and only if $-d^\star \leq L$, and the dual problem attains an optimal solution whenever finite.
When $X$ is empty, it is clear that $p^\star = +\infty$, and by taking both $\vct{\lambda}$ and $\vct{\nu}$ as zero vectors, we have $d^\star = +\infty$.
The result follows.
\end{proof}

Theorem \ref{thm:sigsage_represent_age} is stated in terms of support functions for maximum generality.
From an implementation perspective, it is useful to assume a representation of $X$.
For example, if $X = \{ \vct{x} \,:\, \mtx{A}\vct{x} + \vct{b} \in K \}$ for a matrix $\mtx{A}$, a vector $\vct{b}$, and a convex cone $K$, then weak duality ensures
\[
 \sigma_X(\vct{\lambda}) \doteq \sup\{ \vct{\lambda}^\intercal \vct{x} \,:\, \mtx{A}\vct{x} + \vct{b} \in K \} \leq \inf\{ \vct{b}^\intercal\vct{\eta} \,:\, \mtx{A}^\intercal\vct{\eta} + \vct{\lambda} = \vct{0}, ~ \vct{\eta} \in K^\dagger \}.
\]
An upper bound on the support function is all we need to construct an inner-approximation of a given AGE cone.
For all $X = \{ \vct{x} \,:\, \mtx{A}\vct{x} + \vct{b} \in K \}$, we have
\begin{align*}
	\{ \vct{c} :&~  \vct{\nu} \text{ in } \mathbb{R}^{m-1}, \vct{c} \text{ in } \mathbb{R}^m, \text{ and } \vct{\eta} \text{ in } K^\dagger \nonumber \\
	&~\relent{\vct{\nu}}{\vct{c}_{\setminus i}} - \vct{\nu}^\intercal \vct{1} + \vct{\eta}^\intercal \vct{b} \leq c_i, \nonumber \\
	&~[\mtx{\alpha}_{\setminus i} - \vct{1}\vct{\alpha}_i]^\intercal\vct{\nu} = \mtx{A}^\intercal \vct{\eta},\text{ and } \vct{c}_{\setminus i} \geq \vct{0} \} \subset \cage{\mtx{\alpha},i,X}.
\end{align*}
If there exists an $\vct{x}_0$ so that $\mtx{A}\vct{x}_0 + \vct{b}$ belongs to the relative interior of $K$, then by Slater's condition the reverse inclusion in the preceding expression also holds.

\subsection{Dual perspectives and solution recovery}\label{sec:condsage_sigs:dual}

Here we discuss how dual SAGE relaxations can be used to recover optimal and near-optimal solutions to signomial programs of the form \eqref{eq:background:genericConstrainedMin}.
For concreteness, we state the simplest such relaxation here.
Let $f$, $\{g_i\}_{i=1}^{k_1}$ and $\{ \phi_i\}_{i=1}^{k_2}$ be signomials over exponents $\mtx{\alpha}$, with $\vct{\alpha}_1 = \vct{0}$.
If $\vct{c}$ is the coefficient vector of $f$, and the rows of $\mtx{G} \in \R^{k_1 \times m}$, $\mtx{\Phi} \in \R^{k_2 \times m}$ specify coefficient vectors of $g_i$, $\phi_i$ respectively, then
\begin{equation}
    \inf\{ \vct{c}^\intercal \vct{v} \,: \vct{v} \in \csage{\mtx{\alpha},X}^\dagger,~ v_1 = 1, ~ \mtx{G}\vct{v} \geq \vct{0}, ~ \mtx{\Phi}\vct{v} = \vct{0} \} \label{eq:basic_dual_sage_relaxation}
\end{equation}
is a convex relaxation of Problem \ref{eq:background:genericConstrainedMin}.
It is readily verified that if $\vct{v}^\star$ is an optimal solution to \eqref{eq:basic_dual_sage_relaxation} and $\vct{v}^\star = \exp(\mtx{\alpha}\vct{x})$ for some $\vct{x}$ in $X$, then $\vct{x}$ is optimal for \eqref{eq:background:genericConstrainedMin}.

The prospect of inverting the moment-type vector $\vct{v}^\star$ to obtain a feasible point $\vct{x} \in X$ drives our interest in understanding the dual cone $\csage{\mtx{\alpha},X}^\dagger$.
By standard rules in convex analysis, the dual SAGE cone is given by
\[
\csage{\mtx{\alpha},X}^\dagger = \cap_{i \in [m]} \cage{\mtx{\alpha},i,X}^\dagger.
\]
An expression for the dual AGE cones can be recovered from Theorem \ref{thm:sigsage_represent_age}.
Using $\cone X = \{ (\vct{x}, t) \,:\, t > 0, ~ \vct{x} / t \in X \} $ to denote the cone over $X$, the usual conic-duality calculations yield
\begin{align}
\cage{\mtx{\alpha},i,X}^\dagger = \cl\{ \vct{v} :&~ v_i \log(\vct{v} / v_i) \geq [\mtx{\alpha}- \vct{1}\vct{\alpha}_i] \vct{z} \nonumber \\
&~ (\vct{z}, v_i) \in \cone X, ~\vct{v} \text{ in } \mathbb{R}^m_+, \text{ and } \vct{z} \text{ in } \mathbb{R}^n \}. \label{eq:represent_c_age_i_x_star}
\end{align}
The auxiliary variables ``$\vct{z}$'' appearing in the representation for the dual $X$-AGE cones are a powerful tool for solution recovery.
If $\vct{z}, \vct{v}$ satisfy the constraints in \eqref{eq:represent_c_age_i_x_star} with $v_i > 0$, then $\vct{x} \doteq \vct{z} / v_i$ belongs to $X$.
Additionally, if the inequality constraints involving the logarithm are binding \textit{and} we meet the normalization condition $v_i = \exp( \vct{\alpha}_i \cdot \vct{x})$, then we will have $\exp(\mtx{\alpha}\vct{x}) = \vct{v}$.

These observations form the core -- but not the entirety -- of our solution recovery algorithm.
Depending on solver behavior and dimension-reduction techniques used to simplify a SAGE relaxation, it is possible that $\vct{v} = \exp(\mtx{\alpha}\vct{x})$ for some $\vct{x}$ in $X$, and yet no dual AGE cone generated this value.
Therefore it is also prudent to solve a constrained least-squares problem to find $\vct{x}$ in $X$ with $\mtx{\alpha}\vct{x}$ near $\log \vct{v}$.

\begin{algorithm}[H]
Input: Signomials $f$, $\{g_i\}_{i=1}^{k_1}$, and $\{ \phi_i\}_{i=1}^{k_2}$ over exponents $\mtx{\alpha}$ in $\R^{m \times n}$.
A vector $\vct{v}$ in $\csage{\mtx{\alpha},X}^\dagger$.
Infeasibility tolerances $\epsilon_{\text{ineq}}, \epsilon_{\text{eq}} \geq 0$.
\begin{algorithmic}[1]
\Procedure{SigSolutionRecovery}{$f, g, \phi, \mtx{\alpha}, \vct{v}, \epsilon_{\text{ineq}},\epsilon_{\text{eq}}$}
\State $\mathrm{solutions} \gets []$
\For{$j=1,\ldots,\text{length}(\vct{v})$}
\State Recover $\vct{z}$ in $\R^n$ s.t. $v_j \log(\vct{v} / v_j) \geq [\mtx{\alpha}- \vct{1}\vct{\alpha}_j] \vct{z}~$ and $~(\vct{z}, v_j) \in \cone X$.
\State $\vct{x} \gets \vct{z} / v_j$.
\State $\mathrm{solutions}.\text{append}( \vct{x} )$.
\EndFor
\If{ $\mtx{\alpha}\vct{x} \neq \log \vct{v}$ for all $\vct{x}$ in $\mathrm{solutions}$}
    \State Compute $\vct{x}_{\mathrm{ls}} \inn \text{argmin}\{\|\log \vct{v} - \mtx{\alpha}\vct{x} \| \,:\, \vct{x} \inn X \}$.
    \State $\mathrm{solutions}.\text{append}( \vct{x}_{\mathrm{ls}} )$.
\EndIf
\State $\mathrm{solutions} \leftarrow [~ \vct{x} \inn \mathrm{solutions} ~ \text{ if } ~ g(\vct{x}) \geq -\epsilon_{\text{ineq}} \text{ and } |\phi(\vct{x})| \leq \epsilon_{\text{eq}} ~]$.
\State $\mathrm{solutions}.\text{sort}(f, \text{increasing})$.
\State \textbf{return} $\mathrm{solutions}$.
\EndProcedure
\end{algorithmic}
\caption{signomial solution recovery from dual SAGE relaxations.}
\label{alg:sp_solrec}
\end{algorithm}
Assuming that Equation \ref{eq:represent_c_age_i_x_star} is used to represent the dual AGE cones, the runtime of Algorithm \ref{alg:sp_solrec} is dominated by Line 10.
The runtime of Line 10, in turn, should be substantially smaller than the time required to compute $\vct{v} \in \csage{\mtx{\alpha}}^\dagger$ which is optimal for an appropriate SAGE relaxation.

Note how Algorithm \ref{alg:sp_solrec} implicitly assumes that $\vct{v}$ is elementwise positive.
For numerical reasons, this will be the case in practice.
At a theoretical level, if $\vct{v}$ contains some component $v_i = 0$, then there is no $\vct{x}$ in $\R^n$ which could attain $v_i = \exp(\vct{\alpha}_i \cdot \vct{x})$, and so solution recovery is not a well-posed problem in such situations.
Note also how in the important (and possibly nontrivial) case with $k_1 = k_2 = 0$, Algorithm \ref{alg:sp_solrec} always returns at least one feasible solution.

In the authors experience it is useful to take solutions generated from Algorithm \ref{alg:sp_solrec} and pass them to a local solver as initial conditions.
This additional step is especially worthwhile when there is a gap between a SAGE bound and a problem's true optimal value, or when solution recovery to the desired precision of $\epsilon_{\text{ineq}}, \epsilon_{\text{eq}}$ fails.
The term ``Algorithm \ref{alg:sp_solrec}L'' henceforth refers to the use of Algorithm \ref{alg:sp_solrec}, followed by local-solver solution refinement.
For the examples in this article, the authors use Powell's COBYLA solver for the solution refinement \cite{Powell1994}.\footnote{A FORTRAN implementation is accessible through SciPy's \texttt{optimize} submodule.
The arguments we pass to that FORTRAN implementation are \texttt{RHOBEG}=1, \texttt{RHOEND}$=10^{-7}$, and \texttt{MAXFUN}$=10^5$.}

\subsection{A first worked example}\label{sec:condsage_sigs:ex1}

The following problem has appeared in many articles concerning algorithms for signomial programming \cite{Shen2004,Shen2006,Qu2007,Shen2008,HSC2014}.
\begin{align}
 \inf_{\vct{x} \in \R^3} &~ f(\vct{x}) \doteq 0.5 \exp(x_1 - x_2) -\exp x_1  - 5 \exp(-x_2) \nonumber \tag{Ex1} \\
                        \text{s.t.} &~ g_1(\vct{x}) \doteq 100 -  \exp(x_2 - x_3) -\exp x_2 - 0.05 \exp(x_1 + x_3) \geq 0 \nonumber \\
                                    &~ g_{2:4}(\vct{x}) = \exp \vct{x} - (70,\,1,\, 0.5) \geq \vct{0} \nonumber \\
                                    &~ g_{5:7}(\vct{x}) = (150,\,30,\,21) - \exp \vct{x} \geq \vct{0} \nonumber
\end{align}
This problem (``Example 1'') is an excellent candidate for conditional SAGE relaxations, because each of the seven constraints defines an efficiently representable convex set.
Constraint functions $g_{2:7}$ can be can be represented with six linear inequalities, and the constraint function $g_1$ can be represented with three exponential cones and one linear inequality.
As a separate matter, Example 1 is interesting because the Lagrange dual problem performs poorly:
regardless of how many products we take of existing constraint functions $g_i$, the $-5 \exp(-x_2)$ term in the objective will cause Lagrangians $f - \sum_{I} \lambda_I \prod_{j \in I} g_j$ to be unbounded below for all values of dual variables $\lambda_I \geq 0$.

Now we describe how SAGE relaxations fare for Example 1.
We begin by setting $X = \{ \vct{x} : g_{1:7}(\vct{x}) \}$; since $X$ is bounded, Theorem \ref{thm:sigsage_bounded_error} tells us that $f_X^{\mathrm{SAGE}}$ is finite.
The dual SAGE relaxation can be solved with MOSEK on Machine $\laptop$ in 0.01 seconds, and provides us with a lower bound $f_X^{\mathrm{SAGE}} = -147.86 \leq f_X^\star$.
By running Algorithm \ref{alg:sp_solrec} on the dual solution, we recover
\[
\vct{x}^\star = (5.01063529,~ 3.40119660, -0.48450710) ~~ \text{ satisfying } ~~ f(\vct{x}^\star) = -147.66666.
\]
From this solution, we know that the bound obtained from the SAGE relaxation is within 0.13\% relative error of the true optimal value.
The ability to recover near-optimal solutions even in the presence of a gap $f_{X}^{\mathrm{SAGE}} < f_X^\star$ can be attributed to how our solution recovery algorithm differs from traditional ``moment'' techniques.
As it happens, the point $\vct{x}^\star$ recovered from Algorithm 1 is actually an optimal solution to Example 1.
In order to certify this fact, we need stronger SAGE relaxations. Table \ref{tab:sage_sig_ex_1} shows the results of these relaxations, using the minimax-free hierarchy described in Section \ref{sec:condsage_sigs:ref_hier}.

\begin{table}[h!]
    \centering
    \begin{tabular}{cccc}
    \hline
    level & SAGE bound & $\workstation$ time (s) & $\laptop$ time (s) \\ \hline
    0 & -147.85713 & 0.03 & 0.01 \\
    1 & -147.67225 & 0.05 & 0.02 \\
    2 & -147.66680 & 0.08 & 0.08 \\
    3 & -147.66666 & 0.19 & 0.26 \\ \hline
    \end{tabular}
    \caption{SAGE bounds for Example 1, with solver runtime for Machines $\workstation$ and $\laptop$.
    A level-$3$ bound certifies the level-0 solution as optimal, within relative error $10^{-8}$.}
    \label{tab:sage_sig_ex_1}
\end{table}

\subsection{Reference hierarchies for signomial programming}\label{sec:condsage_sigs:ref_hier}

This section gives a particular set of choices regarding SAGE-based hierarchies for signomial programming.
Because SAGE certificates are sparsity-preserving, one must be careful when describing relaxations which use nonconstant Lagrange multipliers, or positive-definite modulators.
In particular, when we say ``$f$ and $g_i$ are signomials over exponents $\mtx{\alpha}$,'' we mean that $\{ \vct{x} \mapsto \exp(\vct{\alpha}_j \cdot \vct{x}) \}_{j=1}^m$ is the smallest monomial basis spanning all linear combinations of $f$, $g_i$, and the function that is identically equal to one.

First we describe a SAGE-based hierarchy that does not make use of the minimax inequality.
This could be understood as a hierarchy for unconstrained optimization, but really applies whenever minimizing a signomial over a tractable convex set $X \subset \R^n$.
In the event that we cannot certify nonnegativity $f - \gamma$ with  $\gamma = f_X^\star$, we can use modulators as described in Section \ref{sec:background:strengthen} to improve the largest SAGE-certifiable bound on $f$.
Formally, for a signomial $f$ over exponents $\mtx{\alpha}$, a nonnegative integer $\ell$, and a tractable convex set $X$, the \textit{level-$\ell$ SAGE relaxation} for $f_X^\star$ is
\[
f_X^{(\ell)} \doteq \sup\{ \, \gamma \, : \, \Sig(\mtx{\alpha},\vct{1})^\ell (f - \gamma) \text{ is } X\text{-SAGE} \}.
\]
The special case with $\ell = 0$ was introduced earlier in this section as ``$f_{X}^{\mathrm{SAGE}}$.''

Now we consider the case with functional constraints;
let $f$, $g_i$, and $\phi_i$ be signomials over exponents $\mtx{\alpha}$.
SAGE relaxations for the problem of computing $(f,g,\phi)_X^\star$ are indexed by three integer parameters: $p$, $q$, and $\ell$.
Starting from $p \geq 0$ and $q \geq 1$, define $\mtx{\alpha}[p]$ as the matrix of exponent vectors for $\Sig(\mtx{\alpha},\vct{1})^p$, and define $g[q]$ as the set of all products of at-most-$q$ elements of $g$ (similarly define $\phi[q]$).
The SAGE relaxation for $(f,g,\phi)_X^\star$ at level $(p, q, \ell)$ is then
\begin{align}
    (f, g,\phi)_{X}^{(p, q, \ell)} =\sup  ~~ \gamma ~~ \text{s.t.} & ~~ s_h, z_h \text{ are } \text{signomials over exponents } \mtx{\alpha}[p]  \label{eq:sig_sage_hier_def} \\ 
                                                &~~ \mathcal{L} \doteq f - \gamma - \textstyle\sum_{h \in g[q]} s_h \cdot h - \textstyle\sum_{h \in \phi[q]} z_h \cdot h \nonumber \\
                                                &~~ \Sig(\mtx{\alpha}, \vct{1})^\ell \mathcal{L} \text{ is an } X\text{-SAGE signomial}  \nonumber \\
                                                &~~ s_h \text{ are } X\text{-SAGE signomials}. \nonumber
\end{align}
The decision variables in \eqref{eq:sig_sage_hier_def} are $\gamma \in \R$, the coefficient vectors of $\{s_h\}_{h \in g[q]}$, and the coefficient vectors of $\{ z_h \}_{h \in \phi[q]}$.
The most basic level of this hierarchy is $(p,q,\ell) = (0, 1, 0)$.
This corresponds to using scalar Lagrange multipliers ($s_h \geq 0$ and $z_h \in \R$), the original constraints ($g[0] = g, ~ \phi[0] = \phi$), and modulating the Lagrangian by the signomial that is identically equal to 1.
Note that when $p > 0$, the Lagrange multipliers $s_h$ are required to be nonnegative only over $X$, rather than over the whole of $\R^n$.

Once an appropriate SAGE relaxation has been solved, there is the matter of attempting to recover a solution from the dual problem.
Oftentimes a SAGE relaxation produces a tight bound on $(f,g,\phi)_{X}^\star$, and yet no feasible solution can be recovered from Algorithm 1 with reasonable values of $\epsilon_{\text{eq}}$.
Thus we also suggest that one eliminate equality constraints through substitution of variables, when possible.
When it is not possible to eliminate all equality constraints, we recommend allowing large violations of equality constraints in Algorithm 1 (e.g. $\epsilon_{\mathrm{eq}} = 1.0$), and passing the returned values as near-feasible points to a local solver in the manner of Algorithm 1L.\footnote{COBYLA is an excellent example of a solver which suppports infeasible starts.}
This principle also extends to allowing large values of $\epsilon_{\text{ineq}}$ prior to solution-refinement, however the authors find that this is usually not necessary.

\subsection{A second worked example}\label{sec:condsage_sigs:ex2}
This section's example can be found in the 1976 PhD thesis of James Yan \cite{Yan1976}, where it illustrates signomial programming in the service of structural engineering design.
This problem is notable because it is nonconvex even when written in exponential form.
Such signomial programs have received limited attention in the engineering design optimization community, largely due to a lack of reliable methods for solving them.
We restate the problem here (as Example 2) in geometric form.\footnote{The objective and inequality constraint functions are multiplied by $10^4$ for numerical reasons; see equation environment (6.15) on page 106 of \cite{Yan1976} for the original problem statement.} 
\begin{align}
 \inf_{\substack{\vct{A} \in \R^3_{++} \\ P \in \R_{++} }} 
            &~ 10^4 (A_1 + A_2 + A_3) \tag{Ex2}  \\
\text{s.t.} &~ 10^4 + 0.01 A_1^{-1}A_3^{} - 7.0711 A_1^{-1} \geq 0  \nonumber \\
            &~ 10^4 + 0.00854 A_1^{-1}P - 0.60385(A_1^{-1} + A_2^{-1}) \geq 0 \nonumber \\
            &~ 70.7107 A_1^{-1} - A_1^{-1}P - A_{3}^{-1}P = 0 \nonumber \\
            &~ 10^4 \geq 10^4 A_1 \geq 10^{-4} \qquad  10^4 \geq 10^4 A_2 \geq 7.0711 \nonumber \\
            &~ 10^4 \geq 10^4 A_3 \geq 10^{-4} \qquad 10^4 \geq 10^4 P_{~} \geq 10^{-4} \nonumber 
\end{align}
Let $X \subset \R^4$ be the feasible set cut out by the eight bound constraints in Example 2.
With an $X$-SAGE relaxation where all constraints appear in the Lagrangian, we obtain $(f,g,\phi)_X^{(0,1,0)} = 14.1423$ in 0.04 seconds of solver time.
This bound is very close to the optimal value claimed by Yan \cite{Yan1976}.
However, Algorithm 1 only returns candidate solutions ``$\vct{x}$'' with equality constraint violations $\phi(\vct{x}) \approx 70$.

In an effort to improve our chances of solution recovery,  we use the equality constraint to \textit{define} the value $P \leftarrow 70.7107 A_3 / (A_3 + A_1)$.
After clearing the denominator $(A_3 + A_1)$ for inequality constraints involving $P$, we obtain a signomial program (in geometric-form) in only the variables $A_1, A_2, A_3$.
We solve a level-$(0, 1, 0)$ dual conditional SAGE relaxation for this signomial program, and exponentiate the result of Algorithm \ref{alg:sp_solrec} to recover
\[
\vct{A} = (7.0711000\mathtt{e-}04,~ 7.0711000\mathtt{e-}04, ~ 1.00000000\mathtt{e-}08), ~  P = \frac{70.7107 A_3}{A_1 + A_3}.
\]
This solution is feasible up to machine precision, and attains objective matching the 14.142300 SAGE bound.
The entire process of solving the SAGE relaxation and recovering the optimal solution takes less than 0.05 seconds on Machine $\workstation$.

\subsection{Remarks on ``geometric-form'' signomial programming}

By now we have seen signomial programs in both exponential and geometric forms.
The authors have hitherto preferred the exponential form, primarily because it allows us to build upon the substantial theories of convex analysis and convex optimization.
However it is important to acknowledge that from an applications perspective, it is far more common to express signomial programs in geometric form.
Here we briefly present a geometric-form parameterization of conditional SAGE certificates for signomial nonnegativity -- both in effort to appeal to those who are accustomed to working with geometric-form signomial programs, and as a prelude to our discussion on conditional SAGE polynomials.

Geometric form signomials $f(\vct{x}) = \sum_{i=1}^m c_i \vct{x}^{\vct{\alpha}_i}$ are defined at points $\vct{x}$ in $\R^{n}_{++}$, and so it only makes sense to discuss conditional nonnegativity cones for signomials over sets $X \subset \R^n_{++}$.
Henceforth, define
\[
\cnnsgeo{\mtx{\alpha},X} = \{ \vct{c} \,:\,   \textstyle\sum_{i=1}^m c_i \vct{x}^{\vct{\alpha}_i} \geq 0 \foralll \vct{x} \inn X\}
\]
for matrices $\mtx{\alpha}$ in $\R^{m \times n}$ and sets $X$ contained in $\R^n_{++}$.
By applying the change of variables $\vct{x} \mapsto \exp \vct{y}$ and considering the subsequent change of domain $X \mapsto \log X$, one may verify that $\cnnsgeo{\mtx{\alpha},X} = \cnns{\mtx{\alpha},\log X}$.
Thus for $X \subset \R^{n}_{++}$, one arrives naturally at the definition
\[
\csagegeo{\mtx{\alpha},X} \doteq \csage{\mtx{\alpha},\log X}.
\]
From here it should be easy to deduce various corollaries for $\csagegeo{\mtx{\alpha},X}$, by appealing to results from Section \ref{sec:condsage_sigs:main}.
The most important such result is that one can efficiently optimize over $\csagegeo{\mtx{\alpha},X}$ whenever $\log X$ is a convex set for which the epigraph of the support function is efficiently representable.

\section{Conditional SAGE certificates for polynomials}\label{sec:condsage_polys}

In the previous section we saw how conditional SAGE certificates for signomial nonnegativity are inextricably linked to convex duality.
Here we show how the broader idea of conditional SAGE certificates can extend to polynomials.
In this context it is not convexity of $X$ that determines when an $X$-SAGE polynomial cone is tractable, but rather convexity of an appropriate logarithmic transform of $X$.

The organization of this section is similar to that of Section \ref{sec:condsage_sigs}.
Definitions, representations, and other basic theorems for the conditional SAGE polynomial cones are given in Section \ref{sec:condsage_polys:basics}.
Section \ref{sec:condsage_polys:solrec} covers solution recovery from dual SAGE relaxations, and Section \ref{sec:condsage_polys:worked_ex_1} provides a worked example with special focus on solution recovery.
Section \ref{sec:condsage_polys:refhier} describes reference hierarchies for optimization with conditional SAGE polynomial cones: one based on the minimax inequality, and one that is ``minmax free.''
Section \ref{sec:condsage_polys:worked_ex_2} applies various manifestations of the minimax hierarchy to an example problem.

\subsection{The conditional SAGE polynomial cones}\label{sec:condsage_polys:basics}

We call $f = \Poly(\mtx{\alpha},\vct{c})$ an \textit{$X$-AGE polynomial} if it is nonnegative over $X$, and $f(\vct{x})$ contains at most one term $c_i \vct{x}^{\vct{\alpha}_i}$ which is negative for some $\vct{x}$ in $X$.

\begin{definition}[Conditional AGE polynomial cones]\label{def:poly_cond_age}
For $\mtx{\alpha}$ in $\mathbb{N}^{m \times n}$, a subset $X$ of $\R^n$, and an index $i$ in $[m]$, the $i^{\text{th}}$ \textit{AGE polynomial cone} with respect to $\mtx{\alpha}, X$ is
\begin{align*}
\cpolyage{\mtx{\alpha},i,X} = \{ \vct{c} :& ~ \Poly(\mtx{\alpha},\vct{c})(\vct{x}) \geq 0 \foralll \vct{x} \inn X, \nonumber \\
                                          & ~ c_j \geq 0 ~ \text{ if } ~ j \neq i \text{ and } \vct{x}^{\vct{\alpha}_j} > 0 \text{ for some } \vct{x} \inn X, \\
                                          & ~ c_j \leq 0 ~ \text{ if } ~ j \neq i \text{ and } \vct{x}^{\vct{\alpha}_j} < 0 \text{ for some } \vct{x} \inn X~  \}.
\end{align*}
\end{definition}
Let us work through some consequences of the definition.
For starters, if $\vct{x}^{\vct{\alpha}_j}$ takes on positive and negative values as $\vct{x}$ varies over $X$, then $c_j = 0$ whenever $\vct{c} \in \cpolyage{\mtx{\alpha},i,X}$ and $i \neq j$.
Note that $\vct{x}^{\vct{\alpha}_j}$ can only take on both positive and negative values when $\vct{\alpha}_j$ does not belong to the even integer lattice.
If $X$ contains an open ball around the origin, then $\vct{x}^{\vct{\alpha}_j}$ takes on both positive and negative values \textit{if and only if} $\vct{\alpha}_j$ does not belong to the even integer lattice.
Thus Definition \ref{def:poly_cond_age} agrees with the definition of ordinary AGE polynomial cones, as proposed in \cite{SAGE2} and as restated in Equation \ref{eq:def_poly_age_ordinary}.
Another important case is when $X$ is a subset of the nonnegative orthant.
This point is addressed in some detail later in this section; as a preliminary remark, we note that by considering the connection between polynomials and geometric-form signomials, one can easily see that if $X \subset \R^n_{++}$ then $\cpolyage{\mtx{\alpha},i,X} = \cage{\mtx{\alpha},i,\log X}$.
With these facts in mind, we define the conditional SAGE polynomial cone in the natural way.

\begin{definition}[$X$-SAGE polynomials]\label{def:poly_Y_sage}
If the vector $\vct{c}$ belongs to
\[
\cpolysage{\mtx{\alpha}, X} \doteq \sum_{k=1}^m \cpolyage{\mtx{\alpha},k,X}
\]
then we say that $f = \Poly(\mtx{\alpha},\vct{c})$ is an $X$-SAGE polynomial.
\end{definition}

Many of our earlier theorems for signomials apply to $X$-SAGE polynomials without any special assumptions on $X$.
For example, it is easy to show that Theorem \ref{thm:sigsage_bounded_error} applies to conditional SAGE polynomials:
if $X$ is bounded, then $f = \Poly(\mtx{\alpha}, \vct{c})$ with $\vct{\alpha}_1 = \vct{0}$ has
\[
f_{X}^{\mathrm{SAGE}} \doteq \sup\{ \, \gamma \,:\, \gamma \inn \R,~ \vct{c} - \gamma (1,0,\ldots,0)  \inn \cpolysage{\mtx{\alpha},X} \} > -\infty.
\]
Corollary \ref{cor:sigsage_sparsity} likewise extends to polynomials.
Other than substituting AGE signomial cones with AGE polynomial cones, the only difference is that $\mathcal{N}$ becomes $\mathcal{N} = \{ i : c_i\vct{x}^{\vct{\alpha}_i} < 0  \text{ for some } \vct{x} \inn X \}$.

Now we turn to representation of SAGE polynomial cones.
By applying a simple continuity argument one can show that if $X = \cl X^\circ \subset \R^n_+$ -- where $X^\circ$ is the interior of $X$ -- then $\cpolysage{\mtx{\alpha},X} = \csage{\mtx{\alpha},\log X^\circ}$.
This claim is strengthened slightly and made more explicit through the following theorem.
\begin{theorem}\label{thm:reduce_polysage_to_sigsage_2}
Suppose $X \subset \R^n_+$ is representable as $X = \cl\{ \vct{x} : \vct{0} < \vct{x}, H(\vct{x}) \leq \vct{1} \}$ for a continuous map $H : \R^n \to \R^r$. Then for $Y = \{ \vct{y} : H(\exp \vct{y}) \leq \vct{1}\}$, we have $\cpolysage{\mtx{\alpha},X} =\csage{\mtx{\alpha},Y}$.
\end{theorem}
The proof of Theorem \ref{thm:reduce_polysage_to_sigsage_2} is straightforward, and hence omitted.
A more sophisticated result concerns when $X$ is not contained in any particular orthant, but nevertheless possesses a certain sign-symmetry.
\begin{theorem}\label{thm:reduce_polysage_to_sigsage}
Suppose $X \subset \R^n$ is representable as $X = \cl\{ \vct{x} : \vct{0} < |\vct{x}|, H(|\vct{x}|) \leq \vct{1} \}$ for a continuous map $H : \R^n \to \R^r$. Then for $Y = \{ \vct{y} : H(\exp \vct{y}) \leq \vct{1}\}$, we have
\begin{align}
\cpolysage{\mtx{\alpha},X} = \{ \vct{c} ~:&~  \sigreps{\mtx{\alpha},\vct{c}} \cap \csage{\mtx{\alpha}, Y} \text{ is nonempty } \}. \label{eq:reducepolysagetosigsage}
\end{align}
\end{theorem}

By combining Theorem \ref{thm:sigsage_represent_age} with Theorems \ref{thm:reduce_polysage_to_sigsage_2} and \ref{thm:reduce_polysage_to_sigsage}, we know that there exist a range of sets $X$ for which optimization over $X$-SAGE polynomials is tractable.
There remains the potentially nontrivial task of formulating a problem so that one of these theorems provides an efficient representation of $\cpolysage{\mtx{\alpha}, X}$;
important examples of when this is possible include constraints such as
\[
-a \leq x_j \leq a, \qquad \| \vct{x} \|_p \leq a, \qquad |\vct{x}^{\vct{\alpha}_i}| \geq a, \quad \text{ and } \quad x_j^2 = a
\]
where $a > 0$ is a fixed constant.

\begin{proof}[Proof of Theorem \ref{thm:reduce_polysage_to_sigsage}]
Suppose that $\vct{c}$ in $\cpolysage{\mtx{\alpha},X}$ admits the decomposition $\vct{c} = \sum_{i=1}^m \vct{c}^{(i)}$, where $\vct{c}^{(i)}$ belongs to the $i^{\text{th}}$ AGE polynomial cone with respect to $\mtx{\alpha}, X$.
Define $\{ \vct{\tilde{c}}^{(i)} \}_{i=1}^m$ as follows
\[
 \vct{\tilde{c}}^{(i)}_j = \begin{cases} -|c^{(i)}_j| &\text{ if } \vct{\alpha}_i \text{ is not even, and } j = i \\
                                         c^{(i)}_j &\text{ if } \text{ otherwise }
                            \end{cases}.
\]
By the invariance of $X$ under reflection about hyperplanes $\{ \vct{x} : x_j = 0 \}$, and continuity of polynomials, we have that
\begin{align*}
    0 \leq \inf\{\, \Poly(\mtx{\alpha},\vct{c}^{(i)})(\vct{x}) \,:\, \vct{x} \inn X \} 
        &= \inf\{\, \Poly(\mtx{\alpha},\vct{\tilde{c}}^{(i)})(\vct{x}) \,:\, \vct{x} \inn X \cap \R^n_+ \} \\
        &= \inf\{\, \Sig(\mtx{\alpha},\tilde{\vct{c}}^{(i)})(\vct{y}) \,:\, \vct{y} \inn Y \}.
\end{align*}
The signomials $\Sig(\mtx{\alpha},\vct{\tilde{c}}^{(i)})$ are thus nonnegative over $Y = \{ \vct{y} : H(\exp \vct{y}) \leq 1\}$, and posses at most one negative coefficient.
This implies that $\vct{\tilde{c}} \doteq \sum_{i=1}^m \vct{\tilde{c}}^{(i)}$ belongs to $\csage{\mtx{\alpha},Y}$.
One may verify that $\vct{\tilde{c}}$ also satisfies $\tilde{\vct{c}} \in \sigreps{\mtx{\alpha},\vct{c}}$, and so we conclude that the right-hand-side of Equation \eqref{eq:reducepolysagetosigsage} contains $\cpolysage{\mtx{\alpha},X}$.

Now we address the reverse inclusion.
Let $\vct{c}$ be such that $\sigreps{\vct{\alpha},\vct{c}} \cap \csage{\mtx{\alpha},Y}$ is nonempty.
One may verify that basic properties of $\csage{\mtx{\alpha},Y}$ and $\sigreps{\vct{\alpha},\vct{c}}$ ensure that if the intersection is nonempty, it contains an element $\vct{\tilde{c}}$ satisfying  $|\vct{c}| = |\vct{\tilde{c}}|$.
Henceforth fix $\vct{\tilde{c}}$ satisfying these conditions.
Next we appeal to a relaxed form of Corollary \ref{cor:sigsage_sparsity}.
Setting $\mathcal{N} = \{ i : \tilde{c}_i \leq 0\}$, there exist vectors $\vct{\tilde{c}}^{(i)}$ satisfying
\[
    \vct{\tilde{c}} = \textstyle\sum_{i \in \mathcal{N}} \vct{\tilde{c}}^{(i)}, \quad 
    \vct{\tilde{c}}^{(i)} \in \cage{\mtx{\alpha},i,Y}, \quad \text{ and } \quad 
    \tilde{c}^{(i)}_j = 0 \text{ for all } i \neq j \text{ in } \mathcal{N}.
\]
Note how the definition of $\sigreps{\mtx{\alpha},\vct{c}}$ ensures that $\mathcal{N} = \{ i : \vct{\alpha}_i \text{ is not even, or } c_i \leq 0\}$.
Thus we define $\vct{c}^{(i)}$ by
\[
c^{(i)}_j = \begin{cases}  (\sgn c_j) |\tilde{c}_j| &\text{ if } \vct{\alpha}_i \text{ is not even, and } j = i \\
                    \tilde{c}^{(i)}_j & \text{ if } \text{ otherwise }
            \end{cases}
\]
so that $\vct{c} = \sum_{i \in \mathcal{N}} \vct{c}^{(i)}$, and each $\vct{c}^{(i)}$ has the necessary sign pattern for membership in the $i^{\text{th}}$ AGE cone with respect to $\mtx{\alpha}, X$.
Finally, note that
\[
\inf\{ \Poly(\mtx{\alpha},\vct{c}^{(i)})(\vct{y}) : \vct{x} \inn X\} = \inf\{ \Sig(\mtx{\alpha},\vct{\tilde{c}}^{(i)})(\vct{y}) : \vct{y} \inn Y \} \geq 0.
\]
to complete the proof.
\end{proof}

\subsection{Solution recovery for sparse moment problems}\label{sec:condsage_polys:solrec}

This section concerns using dual SAGE relaxations to recover solutions to optimization problems of the form \eqref{eq:background:ineqConstrainedMin}, where $f$ and $g_i$ are polynomials over exponents $\mtx{\alpha}$. 
If $\mtx{G}$ a $k \times m$ matrix whose $i^{\text{th}}$ row is the coefficient vector of $g_i$, $\vct{\alpha}_1$ is the zero vector, and $\vct{c}$ is the coefficient vector of $f$, then the simplest such relaxation is
\begin{equation}
\inf\{ \vct{c}^\intercal\vct{v} \,:\,  \vct{v} \in \cpolysage{\mtx{\alpha},X}^\dagger,~ v_1 = 1,~ \mtx{G}\vct{v} \geq \vct{0} \}. \label{eq:dual_sage_poly_prob_ex}
\end{equation}
Overall, our goal is to recover vectors $\vct{x}$ in $X$ satisfying $\vct{v} = (\vct{x}^{\vct{\alpha}_1},\ldots,\vct{x}^{\vct{\alpha}_m})$, where $\vct{v}$ is an optimal solution to a relaxation such as the one above.
We assume that $X$ is of a form where one of Theorems \ref{thm:reduce_polysage_to_sigsage_2} or \ref{thm:reduce_polysage_to_sigsage} provide a tractable representation of $\csage{\mtx{\alpha},X}$;
this assumption allows us to leverage the following corollary.
\begin{corollary}\label{cor:rep_dual_sage_poly}
Fix $Y = \{ \vct{y} : H(\exp \vct{y}) \leq \vct{1} \}$ for a continuous $H : \R^n \to \R^r$.
\begin{itemize}
    \item If $X = \cl\{ \vct{x} \,:\, \vct{0} < \vct{x}, ~ H(\vct{x}) \leq \vct{1} \}$, then $\cpolysage{\mtx{\alpha},X}^\dagger = \csage{\mtx{\alpha},Y}^\dagger$.
    \item If $X = \cl\{ \vct{x} \,:\, \vct{0} < |\vct{x}|, ~H(|\vct{x}|) \leq \vct{1}\}$, then
\begin{align*}
    \cpolysage{\mtx{\alpha}, X}^\dagger = \{ \vct{v} :&~ \text{there exists }\vct{\hat{v}} \inn \csage{\mtx{\alpha},Y}^\dagger \text{ with}  \nonumber \\
                                              & ~ |\vct{v}| \leq \vct{\hat{v}}, \text{ and } v_i = \hat{v}_i \text{ when } \vct{\alpha}_i \in 2\mathbb{N}^{1 \times n}  \}. 
\end{align*}
\end{itemize}
\end{corollary}
\noindent We make a running assumption that ``$Y$'' is convex.

Solution recovery for polynomial optimization is more difficult than for signomial optimization, because monomials possess both signs and magnitudes.
We propose a two-phase approach for this problem, where different techniques are used to recover variable magnitudes and variable signs.
The main ideas for each phase are described in Sections \ref{sec:condsage_polys:solrec:mags} and \ref{sec:condsage_polys:solrec:signs}, while the formal algorithms are given in the appendix.
The recovered signs and magnitudes are then combined in an elementary way, as given by the following algorithm.

\begin{algorithm}[H]
Input: Polynomials $f$, $\{g_i\}_{i=1}^{k_1}$, and $\{ \phi_{i}\}_{i=1}^{k_2}$ over exponents $\mtx{\alpha} \in \N^{m \times n}$. Vectors $\vct{v} \in \cpolysage{\mtx{\alpha},X}^\dagger$ and $\vct{\hat{v}} \in \csage{\mtx{\alpha},Y}^\dagger$. Tolerances $\epsilon_{\text{ineq}}, \epsilon_{\text{eq}}, \epsilon_{0} > 0$.
\begin{algorithmic}[1]
\Procedure{PolySolutionRecovery}{$f, g, \phi, \mtx{\alpha}, \vct{v}, \vct{\hat{v}}, \epsilon_{\text{ineq}}, \epsilon_{\text{eq}}, \epsilon_0$}
\State $M \gets \mathrm{VariableMagnitudes}(\mtx{\alpha},~\vct{v},~\vct{\hat{v}},~\epsilon_0)$. \quad \# Algorithm \ref{alg:pop_magrec}
\State $S \gets \{ \vct{1} \}$
\If{$X$ is not a subset of $\R^n_+$}
    \State $S$.union( $\mathrm{VariableSigns}(\mtx{\alpha},~\vct{v})$ ) \quad \# Algorithm \ref{alg:pop_signrec}
\EndIf
\State $\mathrm{solutions} \gets []$.
\For{$\vct{x}_{\mathrm{mag}} $ in $M$ and $\vct{s}$ in $S$}
    \State $\vct{x} \leftarrow \vct{x}_{\mathrm{mag}} \odot \vct{s}$  \quad \# denotes elementwise multiplication
    \If{$g(\vct{x}) \geq -\epsilon_{\text{ineq}}$ and $|\phi(\vct{x})| \leq \epsilon_{\text{eq}}$}
        \State $\mathrm{solutions}$.append($\vct{x}$)
    \EndIf
\EndFor
\State $\mathrm{solutions}.\text{sort}(f,~ \text{increasing})$.
\State \textbf{return} $\mathrm{solutions}$.
\EndProcedure
\end{algorithmic}
\caption{solution recovery for dual SAGE polynomial relaxations.}
\label{alg:pop_solrec}
\end{algorithm}

If $\vct{v}$ is optimal for Problem \eqref{eq:dual_sage_poly_prob_ex} and $\vct{v} = (\vct{x}^{\vct{\alpha}_1},\dots,\vct{x}^{\vct{\alpha}_m})$ for an elementwise nonzero $\vct{x}$ in $X$, then Algorithm \ref{alg:pop_solrec} will return an optimal solution to Problem \ref{eq:background:ineqConstrainedMin}.
As with Algorithm \ref{alg:sp_solrec} in the signomial case, the authors find it useful to apply a simple local solver to the output of Algorithm \ref{alg:pop_solrec} as a sort of solution refinement.
We say ``Algorithm \ref{alg:pop_solrec}L'' in reference to such a method, with COBYLA as the local solver.

\subsubsection{Recovering variable magnitudes}\label{sec:condsage_polys:solrec:mags}

Given $\vct{v}$ in $\cpolysage{\mtx{\alpha},X}^\dagger$, we want to find $\vct{x} \in X$ satisfying $(\vct{x}^{\vct{\alpha}_1},\ldots,\vct{x}^{\vct{\alpha}_m}) = |\vct{v}|$.

Regardless of whether $X$ is sign-symmetric or $X \subset \R^n_+$, the variable $\vct{v} \in \cpolysage{\mtx{\alpha},X}$ is associated with an auxiliary variable $\vct{\hat{v}}$ in $\csage{\mtx{\alpha},Y}$, and the variable $\vct{\hat{v}}$ is associated with additional auxiliary variables $\vct{z}_i$ as part of the dual $Y$-AGE signomial cones.
As we discussed in Section \ref{sec:condsage_sigs:dual}, the vectors $\vct{y}_i = \vct{z}_i / \hat{v}_i$ belong to $Y$, and so the vectors $\vct{x}_i = \exp \vct{y}_i$ must belong to $X$.
These vectors $\vct{x}_i$ are not only feasible with respect to $X$, but also satisfy 
\begin{equation}
(\vct{x}_i^{\vct{\alpha}_1},\ldots,\vct{x}_i^{\vct{\alpha}_m}) = \vct{\hat{v}} \label{eq:condsage_polys:solrec:mags:v_hat_eq}
\end{equation}
under the binding-constraint and normalization conditions discussed in Section \ref{sec:condsage_sigs:dual}.
Of course, if $\vct{\hat{v}} = |\vct{v}|$, then Equation \eqref{eq:condsage_polys:solrec:mags:v_hat_eq} is precisely what we desire from our variable magnitudes.
Since $\vct{\hat{v}} = |\vct{v}|$ always holds at least for $X \subset \R^n_+$, the vectors $\vct{x}_i = \exp(\vct{z}_i / \hat{v}_i)$ are reasonable candidates for variable magnitudes.

When $X$ is sign-symmetric, it is possible that $\vct{\hat{v}}$ does not equal $|\vct{v}|$.
This is particularly likely when $\vct{v}$ is subject to additional linear constraints, such as $\mtx{G}\vct{v} \geq \vct{0}$.
Therefore when $X$ is sign-symmetric it is worth considering variable magnitudes which supplement the ones described above.
We propose that one picks a threshold $\epsilon_0 > 0$, computes
\begin{align}
 \vct{y} \in \text{argmin}\{ \textstyle\sum_{i : v_i \neq 0} (\vct{\alpha}_i \cdot \vct{y} - \log |v_i|)^2 \,:\,
    & \vct{y} \inn Y, \label{eq:condsage_polys:solrec:mags:opt_prob} \\
    & \vct{\alpha}_i \cdot \vct{y} \leq \log(\epsilon_0) \foralll v_i = 0 \} \nonumber
\end{align}
and exponentiates $\vct{x} = \exp \vct{y}$.
The role of $\epsilon_0$ is to ensure that $\vct{x} = \exp \vct{y}$ satisfies $|\vct{x}|^{\vct{\alpha}_i} \leq \epsilon_0$ whenever $v_i = 0$.
Extremely small values of $\epsilon_0$ (such as $10^{-100}$) would be reasonable here.

A formal statement of our method for magnitude recovery (Algorithm \ref{alg:pop_magrec}) can be found in the appendix.

\subsubsection{Recovering variable signs}\label{sec:condsage_polys:solrec:signs}

For $\vct{v}$ in $\R^m$, let $\mtx{\alpha}^{-1}(\vct{v})$ denote the set of $\vct{x} \in \R^n$ satisfying $\vct{v} = (\vct{x}^{\vct{\alpha}_1},\ldots,\vct{x}^{\vct{\alpha}_m})$.
Henceforth, fix $\vct{v}$ and assume $\mtx{\alpha}^{-1}(\vct{v})$ is nonempty.
Here we describe how to find vectors $\vct{s}$ in $\{+1,0,-1\}^n$ so that at least one $\vct{x} \in \mtx{\alpha}^{-1}(\vct{v})$ satisfies $x_i > 0$ when $s_i = +1$, $x_i = 0$ when $s_i = 0$, and $x_i < 0$ when $s_i = -1$.
Once we describe this process, we relax the problem slightly so that $s_i = +1$ allows $x_i = 0$.

First we address when $s_i$ should equal zero.
Let $U = \{ i \in [m] \,:\, v_i \neq 0 \}$.
Consider how if some $\vct{x} \in \mtx{\alpha}^{-1}(\vct{v})$ has $x_j = 0$, then we must have $\alpha_{ij} = 0$ for all $i$ in $U$ (else the equality $\vct{x}^{\vct{\alpha}_i} = v_i \neq 0$ would fail).
Thus when $\alpha_{ij} = 0$ for all $i$ in $U$, we set $s_j = 0$ without loss of generality.
Now let $W = \{ j \in [n] \,:\, \alpha_{ij} > 0 \text{ for some } i \inn U \} $; these are indices for which $s_j$ is not yet decided.
Consider the vector $(\vct{v} < 0) \in \{0, 1\}^n$ with values $(\vct{v} < 0)_i = 1$ if $v_i < 0$, and zero if otherwise.
Let $\mtx{\alpha}[U,:]$ be the submatrix of $\mtx{\alpha}$ formed by rows $\{\vct{\alpha}_i\}_{i \in U}$, and similarly index $(\vct{v} < 0)$.
Finally, solve
\begin{equation}
\mtx{\alpha}[U,:]\vct{z} \equiv (\vct{v} < 0)[U] \mod 2 \quad \text{ and } \quad z_j = 0 \foralll j \inn [n] \setminus W\label{eq:sign_pattern_lin_sys}
\end{equation}
for $\vct{z}$ in $\{0, 1\}^n$.
The remaining $(s_j)_{j \in W}$ are $s_j = -1$ if $z_j = 1$ and $s_j = 1$ otherwise.

An individual solution to \eqref{eq:sign_pattern_lin_sys} can be computed efficiently by Gaussian elimination over the finite field $\mathbb{F}_2$.
Standard techniques for finite-field linear algebra also allow us to compute a basis for a null space of a matrix in mod 2 arithmetic (c.f. \cite{Bard2009}), and so in principle one can readily recover all possible solutions to \eqref{eq:sign_pattern_lin_sys}.
In practice we must be careful, since the number of solutions to the linear system can easily be exponentially large in $n$ (for example, when $\vct{\alpha} \equiv \vct{0}_{n \times m} \mod 2$).
Our formal algorithm for sign recovery accounts for this fact, and employs an additional hueristic to handle the case when \eqref{eq:sign_pattern_lin_sys} is inconsistent.
See the appendix for details.

\subsection{A first worked example}\label{sec:condsage_polys:worked_ex_1}

This section's example is to minimize a function appearing in the formulation of the cyclic $n$-roots problem.
The general cyclic $n$-roots problem is a challenging benchmark problem in computer algebra \cite{PHCpack}.
Our problem is to minimize
\begin{equation}
f(\vct{x}) = -64\sum_{i=1}^7\prod_{j \in [7] \setminus \{ i \}} x_j  \tag{Ex3}
\end{equation}
over the box $X = [-1/2, 1/2]^7$.  To the authors' knowledge, this problem was first used as an optimization benchmark in the work by Ray and Nataraj, on computing the extrema of polynomials over boxes \cite{RN2008}.
One may verify that $f_X^\star = -7$, and that this objective value is attained at $\vct{x}^{(1)} = \vct{1}/2$ and $\vct{x}^{(2)} = -\vct{1}/2$.
Despite this problem's simplicity, it requires nontrivial computational effort with SOS methods.
The lowest relaxation order that allows Gloptipoly3 \cite{gloptipoly} to compute $f_X^\star = -7$ results in a semidefinite program that takes MOSEK 90 seconds to solve with Machine $\workstation$.

SAGE relaxations automatically exploit the structure in this problem.
Since the seven functions $f_i(\vct{x}) = 1-64\prod_{j \neq i} x_i$ are $X$-AGE and sum to $f + 7$, we have that $-7 \leq f_{X}^{\mathrm{SAGE}} \leq f_X^\star$.
To address the dual SAGE relaxation and solution recovery, we introduce the $8 \times 7$ matrix $\mtx{\alpha}$, with final row $\mtx{\alpha}_8 = \vct{0}$, $\alpha_{ii} = 0$ for $i \leq 7$, and $\alpha_{ij} = 1$ for the remaining entries.
Next we write $X =\{ \vct{x} : \vct{x}^2 \leq \vct{1}/4 \}$, and for $Y = \{ \vct{y} : \exp(2\vct{y}) \leq \vct{1}/4 \}$ numerically solve
\[
f_X^{\mathrm{SAGE}} = \inf\{  -64 \cdot \vct{1}^\intercal \vct{v}_{1:7} \,:\, - \vct{\hat{v}} \leq \vct{v} \leq \vct{\hat{v}}, ~ \vct{\hat{v}} \inn \csage{\mtx{\alpha},Y}^\dagger,~ v_{8} = \hat{v}_{8} = 1 \} = -7.
\]
MOSEK solves this problem in 0.01 seconds with Machine $\workstation$.

We recover candidate magnitudes by using the eight $Y$-AGE cones associated with the auxiliary variable $\vct{\hat{v}} \in \csage{\mtx{\alpha},Y}^\dagger$.
To machine precision, each of these AGE cones yields the same candidate magnitude $|\vct{x}| = \vct{1}/2$.
The optimal moment vector $\vct{v} = \vct{1} / 64$ is elementwise positive, and so sign-pattern recovery is a matter of finding all solutions to the system
$\mtx{\alpha} \vct{z} \equiv \vct{0} \mod 2$. 
There are exactly two solutions to this system: $\vct{z}^{(1)} = \vct{0}$, and $\vct{z}^{(2)} = \vct{1}$.
The first of these gives rise to signs $\vct{s}^{(1)} = \vct{1}$, and the second of these results in $\vct{s}^{(2)} = -\vct{1}$.
By combining these candidate signs with candidate magnitudes, we obtain candidate solutions $\{ \vct{1}/2, - \vct{1}/2 \}$; since these solutions are feasible and obtain objective values matching the SAGE bound, we conclude that both candidate solutions are minimizers of $f$ over $X$.

\subsection{Reference hierarchies for POPs}\label{sec:condsage_polys:refhier}

If $X \subset \R^n_+$, then one should use the same hierarchies described in Section \ref{sec:condsage_sigs:ref_hier}, where ``$\Sig$'' is replaced by ``$\Poly$'' and constraints that a function is ``an $X$-SAGE signomial'' are replaced by constraints that the function is ``an $X$-SAGE polynomial.''
This section focuses on the more complicated case when $X$ is sign-symmetric.

Our reference hierarchy for functionally constrained polynomial optimization is similar to that used for signomial programming.
Let $f$, $\{ g_i \}_{i=1}^{k_1}$, and $\{ \phi_i\}_{i=1}^{k_2}$ be polynomials over common exponents $\mtx{\alpha} \in \N^{m \times n}$, and fix sign-symmetric $X \subset \R^n$.
Define $\mtx{\hat{\alpha}}$ as the matrix formed by stacking $\mtx{\alpha}$ on top of $2\mtx{\alpha}$, and then removing any duplicate rows.
The SAGE relaxation for $(f,g,\phi)_X^\star$ at level $(p, q, \ell)$ is then
\begin{align}
    (f, g,\phi)_{X}^{(p, q, \ell)} =\sup  ~~ \gamma ~~ \text{s.t.} & ~~ s_h, z_h \text{ are } \text{polynomials over exponents } \mtx{\hat{\alpha}}[p] \label{eq:poly_sage_hier_def}  \\ 
                                                &~~ \mathcal{L} \doteq f - \gamma - \textstyle\sum_{h \in g[q]} s_h \cdot h - \textstyle\sum_{h \in \phi[q]} z_h \cdot h \nonumber \\
                                                &~~ \Poly(2\mtx{\alpha}, \vct{1})^\ell \mathcal{L} \text{ is an } X\text{-SAGE polynomial}  \nonumber \\
                                                &~~ s_h \text{ are } X\text{-SAGE polynomials}. \nonumber
\end{align}
As before, the decision variables are $\gamma \in \R$, and the coefficient vectors of $\{s_h\}_{h \in g[q]}$, $\{ z_h \}_{h \in \phi[q]}$.
The main difference between \eqref{eq:poly_sage_hier_def} and it's signomial equivalent \eqref{eq:sig_sage_hier_def}, is that the Lagrange multipliers are slightly more complex in \eqref{eq:poly_sage_hier_def}. This change was made to improve performance for problems where only a few rows of $\mtx{\alpha}$ belong to the nonnegative even integer lattice.

Our minimax-free reference hierarchy for polynomial optimization is meaningfully different from the signomial case.
We begin by assuming a representation $X = \cl\{ \vct{x} \,:\, \vct{0} < |\vct{x}|, ~ H(|\vct{x}|) \leq \vct{1} \}$, and subsequently defining $Y = \{ \vct{y} \,:\, H(\exp \vct{y}) \leq \vct{1} \}$.
Let $\mathcal{A}$ and $\mathcal{C}$ be operators on polynomials so that $f = \Poly(\mathcal{A}(f),\mathcal{C}(f))$ always holds, and
let $\vct{s}$ be the vector in $\R^m$ with $s_i = 1$ when $\vct{\alpha}_i$ is even, and $s_i = 0$ otherwise.
The SAGE relaxation for $f_X^\star$ at level $(p, q)$ is
\begin{align}
    f_X^{(p, q)} = \sup & ~~ \gamma \label{eq:poly_sage_uncon_hier} \\
    \text{s.t.} & ~~ \psi \doteq \Poly(\mtx{\alpha},\vct{s})^{p}(f - \gamma) \nonumber \\
    & ~~ \vct{c} \in \sigreps{\mathcal{A}(\psi),\, \mathcal{C}(\psi)} \nonumber \\
    & ~~ [\Sig(\mathcal{A}(\psi),
    \,\vct{1})]^{q} \Sig(\mathcal{A}(\psi),~\vct{c}) \text{ is } Y\text{-SAGE} \nonumber
\end{align}
over optimization variables $\vct{c}$ and $\gamma$.

Formulation \eqref{eq:poly_sage_uncon_hier} uses two parameters out of desire to mitigate \textit{both} sources of error in the SAGE polynomial cone: the error from replacing a polynomial by its signomial representative, and the error from replacing the signomial nonnegativity cone by the SAGE cone.
As we show in Section \ref{sec:comp:poly_ex}, the signomial representative complexity parameter ``$q$'' can make the difference in our ability to solve problems when $X = \R^n$.

\subsection{A second worked example}\label{sec:condsage_polys:worked_ex_2}

The following problem appears in work on ``Bounded Degree Sums-of-Squares'' (BSOS) and ``Sparse Bounded Degree Sums-of-Squares'' (Sparse-BSOS) methods for polynomial optimization \cite{LTY2017,WLT2018}.
The latter paper reports that BSOS and Sparse-BSOS compute $(f,g)_{\R^6}^\star = -0.41288$ in 44.5 and 82.1 seconds respectively, when using SDPT3-4.0 on a machine with a quad-core 2.6GHz Core i7 processor and 16GB RAM.

\begin{align}
 \inf_{\vct{x} \in \R^6} &~ f(\vct{x}) \doteq x_1^6 - x_2^6 + x_3^6 - x_4^6 + x_5^6 - x_6^6 + x_1 - x_2  \tag{Ex4} \\
                        \text{s.t.} &~ g_1(\vct{x}) \doteq 2 x_{1}^{6}+3 x_{2}^{2}+2 x_{1} x_{2}+2 x_{3}^{6}+3 x_{4}^{2}+2 x_{3} x_{4}+2 x_{5}^{6}+3 x_{6}^{2}+2 x_{5} x_{6}  \geq 0 \nonumber \\
                                    &~ g_2(\vct{x}) \doteq 2 x_{1}^{2}+5 x_{2}^{2}+3 x_{1} x_{2}+2 x_{3}^{2}+5 x_{4}^{2}+3 x_{3} x_{4}+2 x_{5}^{2}+5 x_{6}^{2}+3 x_{5} x_{6}  \geq 0 \nonumber \\
                                    &~ g_3(\vct{x}) \doteq 3 x_{1}^{2}+2 x_{2}^{2}-4 x_{1} x_{2}+3 x_{3}^{2}+2 x_{4}^{2}-4 x_{3} x_{4}+3 x_{5}^{2}+2 x_{6}^{2}-4 x_{5} x_{6} \geq 0 \nonumber \\
                                    &~ g_4(\vct{x}) \doteq x_{1}^{2}+6 x_{2}^{2}-4 x_{1} x_{2}+x_{3}^{2}+6 x_{4}^{2}-4 x_{3} x_{4}+x_{5}^{2}+6 x_{6}^{2}-4 x_{5} x_{6} \geq 0 \nonumber \\
                                    &~ g_5(\vct{x}) \doteq x_{1}^{2}+4 x_{2}^{6}-3 x_{1} x_{2}+x_{3}^{2}+4 x_{4}^{6}-3 x_{3} x_{4}+x_{5}^{2}+4 x_{6}^{6}-3 x_{5} x_{6} \geq 0 \nonumber \\
                                    &~ g_{6:10}(\vct{x}) \doteq 1 - g_{1:5}(\vct{x}) \geq  \vct{0} \nonumber \\
                                    &~ g_{11:16}(\vct{x}) = \vct{x} \geq \vct{0} \nonumber
\end{align}
This problem (Example 4) is a good example for conditional SAGE polynomial relaxations, because it allows for several choices in partial dualization.

The simplest choice is to use no partial dualization at all-- simply solve relaxations of the form \eqref{eq:poly_sage_hier_def} with $X = \R^n$.
Indeed, it is possible to solve Example 3 with only these ordinary SAGE certificates, however the necessary level of the hierarchy $(f,g)^{(1,1,0)}_{\R^6} = -0.41288$ requires 101 seconds of solver time on Machine $\workstation$.

A strictly preferable alternative is to apply partial dualization with $X = \R^6_+$.
With this choice of $X$ it is natural to drop the first two and last six constraints from $g$ (all of which will be trivially satisfied), and work with $\hat{g} = g_{3:10}$.
This allows us to compute $(f,\hat{g})^{(1,1,0)}_{X} = -0.41288$ in 3.04 seconds of solver time on Machine $\workstation$, and 4.4 seconds of solver time on Machine $\laptop$.
It is significant that the SAGE relaxation can be solved in this amount of time on Machine $\laptop$, since it is an order of magnitude faster than BSOS solve time reported in \cite{WLT2018}.

The most aggressive choice for partial dualization is $X = \{ \vct{x}  \,:\, \vct{x} \geq \vct{0}, g_{6:7} \geq \vct{0} \}$.
With this choice of $X$ one has the option of using $\hat{g} = (g_{3:5},g_{8:10})$, or $\hat{g} = g_{3:10}$; in the first case Machine $\workstation$ computes $(f,\hat{g})_X^{(1,1,0)} = -0.47121$ in 3.3 seconds, and in the second case Machine $\workstation$ computes $(f,\hat{g})_X^{(1,1,0)} = -0.41288$ in 5.67 seconds.
It is worth emphasizing that even though $g_{6:7}$ were incorporated into $X$, the SAGE bound with Lagrange multiplier complexity $p=1$ improved by including $g_{6:7}$ in the Lagrangian.

\section{Computational experiments}\label{sec:experiments}
This section presents the results of some computational experiments with SAGE relaxations.
Experiments with signomial programs consist of twenty-nine problems drawn from the literature, of which seventeen are solved to optimality (see Section \ref{sec:comp:sig_ex}).
Examples for polynomial optimization include twenty-two problems from the literature (Section \ref{sec:comp:poly_ex}), as well as randomly generated problems (Section \ref{sec:comp:rand_poly}).

All experiments described here were run with the provided \texttt{sageopt} python package.
\texttt{Sageopt} is an extension and refinement of the ``\texttt{sigpy}'' package introduced by the authors in the appendix of \cite{SAGE2}.
In the spirit of Gloptipoly3 \cite{gloptipoly}, \texttt{sageopt} includes its own basic rewriting system to cast a SAGE relaxation into conic forms acceptable by appropriate solvers.
The rewriting system also provides mechanisms for computing constraint violations, analyzing low-level problem data, and constructing a set $X$ (in an appropriate representation) from lists of constraint functions $g$ and $\phi$.
Once problem data $f$, $g$, and $\phi$ are defined, a SAGE relaxation can be constructed and solved in just two lines of code.
Solution recovery similarly requires no more than two lines of code.
\texttt{Sageopt} currently supports the conic solvers ECOS \cite{ecos,expConeThesis} and MOSEK \cite{mosek}.

Unless otherwise stated, experiments were conducted on Machine $\workstation$. All experiments were conducted with the MOSEK solver, using the default solver tolerances.
We note that although Sections \ref{sec:condsage_sigs:ref_hier} and \ref{sec:condsage_polys:refhier} only stated the SAGE relaxations in primal form, these experiments were conducted by symbolically constructing primal and dual problems, and solving them separately from one another.
In order to communicate the quality of these numeric solutions, we generally report ``SAGE bounds'' to the farthest decimal point where the primal and dual objectives agree.

\subsection{Signomial programs from the literature}\label{sec:comp:sig_ex}
The examples in this section were drawn from the PhD thesis of James Yan \cite{Yan1976}, a popular benchmarking paper by Rijckaert and Martens \cite{RM1978}, and the more contemporary works \cite{HSC2014,Xu2014}.
The organization of this section is chronological with respect to these three sources.
Many of the problems considered here can be found elsewhere in the literature, including work by Shen et al. \cite{Shen2004,Shen2006,Shen2008}, Wang and Liang \cite{WL2005} and Qu et al. \cite{Qu2007}.

SAGE recovers best-known solutions for all but six of the twenty nine problems considered here.
For every one of these six problematic examples, numerical issues resulted in solver failures for level-$(p, q, \ell)$ relaxations whenever $p > 0$; the results for these six problems should not be taken as definitive.
For the twenty-three problems where SAGE recovered best-known solutions, there are two important trends we can observe.
First, our solution recovery algorithms are more likely to succeed with a conditional SAGE relaxation than with an ordinary SAGE relaxation, even when the ordinary SAGE relaxation is tight.
Second-- the local solver refinement in Algorithm 1L can help tremendously not only in the presence of suboptimal strictly-feasible initial solutions (Example 8), but also in the presence of both large and small constraint violations (Examples 9 and 6 respectively).
The initial condition from a SAGE relaxation in Algorithm 1L is important; the underlying COBYLA solver can and will return suboptimal solutions if initialized poorly.

\subsubsection{Problems from the PhD thesis of James Yan}

We attempted to solve nine example problems appearing James Yan's 1976 PhD thesis \textit{Signomial programs with equality constraints : numerical solution and applications} \cite{Yan1976}.
This section reproduces two of the six problems which we solved to global optimality via SAGE certificates.
Yan's ``Problem B'' (page 88) and ``Problem C'' (page 89) serve as our Examples 5 and 6 respectively.

\begin{align}
 \inf_{ \vct{x} \in \R^4} &~ f(\vct{x}) \doteq 2 - \exp(x_1 + x_2 + x_3) \tag{Ex5}  \\
                        \text{s.t.} &~ g_1(\vct{x}) \doteq 4 - \exp x_3 - 15 \exp(x_2 + x_3) - 15 \exp(x_3 + x_4) \geq 0  \nonumber \\
                                    &~ g_{2:5}(\vct{x}) \doteq  (1,1,1,2) - \exp \vct{x} \geq \vct{0} \nonumber \\
                                    &~ g_{6:9}(\vct{x}) \doteq \exp \vct{x} - (1,1,1,1)/10  \geq \vct{0} \nonumber \\
                                    &~ \phi_1(\vct{x}) \doteq \exp x_1 + 2 \exp x_2 + 2 \exp x_3 - \exp x_4 = 0 \nonumber
\end{align}
It is possible to quickly compute $(f,g,\phi)_{\R^4}^\star = 1.\overline{925}$ with both conditional and ordinary SAGE certificates, although conditional SAGE certificates exhibit better performance for solution recovery.
Specifically, $(f,g,\phi)^{(1,1,0)}_{\R^4} = 1.92592592$ can be computed in 0.12 seconds, but no solution can be recovered from Algorithm \ref{alg:sp_solrec} unless $\epsilon$ is set to an unacceptably large value of $0.1$.
Instead we set $X = \{ \vct{x} \,:\, g(\vct{x}) \geq \vct{0} \}$, compute $(f,g,\phi)_{X}^{(1,1,0)} = 1.92592593$ in 0.18 seconds, and by running Algorithm \ref{alg:sp_solrec} recover $\vct{x}^\star$ satisfying $g(\vct{x}^\star) >$ 1E-11, $|\phi(\vct{x}^\star)| <$ 1E-8, and $f(\vct{x}^\star) = 1.92592593$.

\begin{align}
 \inf_{\vct{y} \in \R^3_{++}} 
            &~ y_1^{0.6}y^{}_2 + y^{}_2 y_3^{-0.5} + 15.98 y_1^{} + 9.0824 y_2^2 - 60.72625 y_3^{} \tag{Ex6}  \\
\text{s.t.} &~ y_{2}^{-2}y_3^{} - y_1^{} y_2^{-2} -0.48 \geq 0  \nonumber \\
            &~ y_1^{0.5} y_3^{2} -y_1^{0.25} y_3^{} - y_2^2 - 5.75 \geq 0 \nonumber \\
            &~ (1000, 1000, 1000) \geq \vct{y} \geq (0.1, 0.1, 0.1) \nonumber \\
            &~ y_1^2 + 4 y_2^2 + 2 y_3^2 - 58 = 0 \nonumber \\
            &~ y_1^{}y_2^{-1}y_3^{2.5} + y_2^{} y_3^{} - y_2^2 - 16.55 = 0 \nonumber
\end{align}
Setting $X = \{ \vct{x} \in \R^3 \,:\, g(\vct{x}) \geq \vct{0} \}$, we can compute $(f,g_{1:2},\phi)_X^{(0,1,0)} = -320.722913$ in 0.04 seconds.
By running Algorithm 1 with $\epsilon_{\text{ineq}} =$1E-8 and $\epsilon_{\text{eq}} =$1E-6, we recover $\vct{x}$ with objective $f(\vct{x}) = -320.722913$ and that is feasible up to tolerance 8E-7.
We then pass this solution to COBYLA with parameter \texttt{RHOEND}=1E-10, and subsequently recover recover $\vct{x}^\star$ with the same objective, but constraint violation of only 5E-13.

The remaining problems which we solved to optimality were ``Problem A'' on page 60, ``Problem A'' on page 88, ``Problem D'' on page 89, and the problem in equation environment ``(6.15)'' on page 106.
The last of these was introduced in Section \ref{sec:condsage_sigs:ex2} as ``Example 2.''
The problems which we did not solve to optimality were ``Problem B'' on page 61, the problem in equation environment ``(6.29)'' on page 113, and the problem in equation environment ``(6.36)'' on page 120.
In each of these unsolved cases, we encountered solver-failures for level-$(p, q, \ell)$ relaxations whenever $p > 0$.
Therefore the bounds computed for each of these problems were essentially limited to those of Lagrange dual problems, with modest partial dualization.

\subsubsection{Problems from the benchmarking paper of Rijckaert and Martens}

We attempted to solve problems 9 through 18 of the popular signomial-geometric programming benchmark paper by Rijckaert and Martens \cite{RM1978}.
Of these ten problems, seven met with at least moderate success, in that SAGE relaxations produced meaningful lower bounds on a problem's optimal value, and also facilitated recovery of best-known solutions to these problems.
SAGE certificates allow us to certify global optimality for four of these seven problems.
Problem statistics and a qualitative summary of SAGE performance is given in Table \ref{tab:rm1978_summary}.

We reproduce Rijckaert and Martens' problems 10 and 15 as our Examples 7 and 8 respectively; both problems are written in exponential-form.
\begin{align}
     \inf_{\vct{x} \in \R^{3}} 
            &~ f(\vct{x}) \doteq 0.5 \exp(x_1 - x_2) - \exp x_1 - 5 \exp(-x_2)  \tag{Ex7}  \\
\text{s.t.} &~ g_1(\vct{x}) \doteq 100 - \exp(x_2 - x_3) - \exp x_1 - 0.05 \exp(x_1 + x_3) \geq 0 \nonumber \\
            &~ g_{2:4}(\vct{x}) \doteq (100, 100, 100) - \exp \vct{x} \geq \vct{0} \nonumber \\
            &~ g_{5:7}(\vct{x}) \doteq \exp \vct{x} - (1,1,1) \geq \vct{0} \nonumber 
\end{align}
The bound constraints appearing in Example 7 are not included in \cite{RM1978}, however $f$ is unbounded below if we omit them.
The solution proposed in \cite{RM1978} has $\exp \vct{x} = (88.310, 7.454, 1.311)$, and objective value $f(\vct{x}) = -83.06$.
The actual optimal solution has value $-83.25$, and this can be certified by running Algorithm 1 on a dual solution for $f_X^{(3)} = -83.2510$, where $X = \{ \vct{x} \,:\, g(\vct{x}) \geq \vct{0}\}$.
Solving the necessary SAGE relaxation takes 0.1 seconds on Machine $\workstation$.

\begin{align}
\inf_{\vct{x} \in \R^{10}} 
            &~ f(\vct{x}) \doteq 0.05 \exp x_1 + 0.05 \exp x_2 + 0.05 \exp x_3 + \exp x_9 \tag{Ex8}  \\
\text{s.t.} &~ g_1(\vct{x}) \doteq 1 + 0.5 \exp(x_1 + x_4 - x_7) - \exp(x_{10} - x_7) \geq 0 \nonumber \\
            &~ g_2(\vct{x}) \doteq 1 + 0.5 \exp(x_2 + x_5 - x_8) - \exp(x_7 - x_8) \geq 0 \nonumber \\
            &~ g_3(\vct{x}) \doteq 1 + 0.5 \exp(x_3 + x_6 - x_9) - \exp(x_8 - x_9) \geq 0 \nonumber \\
            &~ g_4(\vct{x}) \doteq 1 - 0.25 \exp(-x_{10}) - 0.5 \exp(x_9 - x_{10}) \geq 0  \nonumber \\
            &~ g_5(\vct{x}) \doteq 1 - 0.79681 \exp(x_4 - x_7) \geq 0\nonumber \\
            &~ g_6(\vct{x}) \doteq 1 - 0.79681 \exp(x_5 - x_8) \geq 0\nonumber \\
            &~ g_7(\vct{x}) \doteq 1 - 0.79681 \exp(x_6 - x_9) \geq 0\nonumber
\end{align}
A level (1,1,0) ordinary SAGE relaxation for Example 8 can be solved in 2.8 seconds on Machine $\workstation$; this returns the bound $(f,g)_{\R^{10}}^{(1,1,0)} = 0.2056534$.
When Algorithm \ref{alg:sp_solrec} is run on the dual solution, it returns a point $\vct{x}$ satisfying $f(\vct{x}) \approx 0.38$ and $g(\vct{x}) \geq 0.053$.
However by subsequently running Algorithm \ref{alg:sp_solrec}L, we obtain $\vct{x}^\star$ satisfying $f(\vct{x}^\star) = 0.20565341$ and $g_i(\vct{x}^\star) \geq$ 1E-8 for all $i$ in $[k]$.
We thus conclude that the level-$(1,1,0)$ SAGE relaxation was tight.

\begin{table}[h!]
    \centering
\begin{tabular}{c|cc|cc}
Num. in \cite{RM1978} & $n$ & $k$ & solution quality & optimal? \\ \hline
9 & 2 & 1 & same & unknown \\
10 & 3 & 1 & improved & yes \\
11 & 4 & 2 & same & yes \\
12 & 8 & 4 & same & unknown \\
13 & 8 & 6 & no solution & no \\
14 & 10 & 7 & same & unknown \\
15 & 10 & 7 & same & yes \\
16 & 10 & 7 & same & yes \\
17 & 11 & 8 & no solution & no \\
18 & 13 & 9 & no solution & no \\ \hline
\end{tabular}
    \caption{Columns $n$ and $k$ give number of variables and number of inequality constraints for the indicated problem.
    ``Solution quality'' is ``same'' (resp. ``improved'') if Algorithm \ref{alg:sp_solrec}L returned a feasible solution with objective equal to (resp. less than) that proposed in \cite{RM1978}. Problems 9, 12, and 14 are discussed in Table \ref{tab:rm1978_partial_success}. We encountered solver failures for level-$(1,1,0)$ relaxations of problems 13, 17, and 18.}
    \label{tab:rm1978_summary}
\end{table}

\begin{table}[h!]
\begin{center}
\begin{tabular}{c|cccccc}
\hline
 &  \multicolumn{2}{c}{ SAGE relaxation } & \multicolumn{2}{c}{ Algorithm \ref{alg:sp_solrec} } & \multicolumn{2}{c}{ Algorithm \ref{alg:sp_solrec}L } \\
Num. in \cite{RM1978} & $(p, q, \ell)$ & bound & $f(\vct{x})$ & $\min g(\vct{x})$ & $f(\vct{x})$ & $\min g(\vct{x})$ \\ \hline
9 & (3,3,1) & 11.7 & 12.500 & 0.00438 & 11.9600 & 2.00E-10 \\
12 & (0,2,1) & -6.4 & -5.7677 & 0.00034 & -6.0482 & -5.00E-09 \\
14 & (0,4,0) & 0.7 & 2.5798 & 0.01541 & 1.14396 & -8.00E-09 \\ \hline
\end{tabular}
\caption{Problems for which we did not certify optimality, but nevertheless recovered best-known solutions by using SAGE relaxations.
        Note that Algorithm \ref{alg:sp_solrec} returned strictly-feasible solutions in each of these cases.
        In the next section we present examples where Algorithm \ref{alg:sp_solrec} does not return feasible solutions, and so solution refinement (i.e. Algorithm \ref{alg:sp_solrec}L) becomes more important.}\label{tab:rm1978_partial_success}
\end{center}
\end{table}

\FloatBarrier

\subsubsection{Problems from contemporary sources}

Here we describe our attempts at solving six problems from the 2014 article by Hou, Shen, and Chen \cite{HSC2014}, as well as four problems from the 2014 article by Xu \cite{Xu2014}.
SAGE relaxations are quite successful in this regard: seven of the ten problems are solved to global optimality (verified SAGE bounds), while best-known (but possibly suboptimal) solutions are obtained for the remaining three problems.
Summary results can be found in Tables \ref{tab:contemp_summary} and \ref{tab:contemp_unknown}.
We explicitly reproduce problem \cite{HSC2014}-8 as our Example 9.

\begin{table}[h!]
    \centering
    \begin{tabular}{cc|ccc|ccc}
    \hline
    source & num. & $n$ & $k_1$ & $k_2$ & objective & infeasibility & optimal? \\
    \hline
    \cite{HSC2014} & 1 & 4 & 10 & 0 & 0.7650822 & 0 & yes \\
    - & 2 & 2 & 5 & 0 & 11.964337 & 0 & yes \\
    - & 3 & 3 & 7 & 0 & -147.66666 & 0 & yes \\
    - & 5 & 5 & 16 & 0 & 10122.493 & 4.00E-13 & unknown \\
    - & 7 & 3 & 6 & 0 & -10.363636 & 2.00E-15 & yes \\
    - & 8 & 15 & 37 & 6 & 156.21963 & 4.00E-14 & yes \\
    \hline
    \cite{Xu2014} & 4 & 2 & 1 & 1 & 1.3934649 & 2.00E-10 & yes \\
    - & 5 & 6 & 9 & 4 & -0.3888114 & 5.00E-17 & unknown \\
    - & 6 & 2 & 4 & 2 & 1.1771243 & 4.00E-12 & yes \\
    - & 7 & 6 & 20 & 3 & 10252.790 & 8.00E-14 & unknown \\
    \hline
    \end{tabular}
    \caption{Columns $n$, $k_1$, and $k_2$ specify the number of variables, inequality constraints, and equality constraints in the indicated problem.
    The last three columns specify the objective value and constraint violation of a solution obtained by running Algorithm \ref{alg:sp_solrec}L on the output of a dual SAGE relaxation, as well as a note on whether the objective matched a SAGE bound.
    Problems with ``unknown'' optimality status are described in Table \ref{tab:contemp_unknown}.
    } \label{tab:contemp_summary}
\end{table}

\begin{align}
    \inf_{\vct{y} \in \R^{15}_{++}}
            &~ \textstyle\sum_{i=1}^4 y_{i + 11}(12.62626 - 1.231059 y_i) \tag{Ex9} \\
\text{s.t.} &~ y_{12}-y_{11} \leq 0, \quad y_{11}-y_{12} \leq 50, \quad y_{10}-y_{4} \leq 0 \nonumber \\
            &~ y_{9}-y_{10} \leq 0, \quad y_{8}-y_{9} \leq 0, \quad 2 y_{7}-y_{1} \leq 1 \nonumber \\
            &~ y_{3}-y_{4} \leq 0, \quad y_{2}-y_{3} \leq 0, \quad y_{1}-y_{2} \leq 0 \nonumber \\
            &~ 50 y_{4}+y_{10} y_{15}-50 y_{10} -y_{4} y_{15} \leq 0 \nonumber \\
            &~ 50 y_{10}+y_{4} y_{5}+y_{9} y_{14}-50 y_{9} - y_{3} y_{14}-y_{8} y_{15} \leq 0 \nonumber \\
            &~ 50 y_{7}+y_{2} y_{13}+y_{7} y_{12}-50 y_{8} - y_{1} y_{12}-y_{8} y_{13} \leq 0 \nonumber \\
            &~ 50 y_{8}+y_{1} y_{12}+y_{8} y_{13}-50 y_{7} - y_{2} y_{13}-y_{7} y_{12} \leq 0 \nonumber \\
            &~ 50 y_{8}+50 y_{9}+y_{3} y_{14}+y_{8} y_{13} - y_{2} y_{13}-y_{9} y_{14} \leq 500 \nonumber \\
            &~ y_{6} y_{11}+y_{1} y_{12}+y_{7} y_{11}-y_{6} y_{12} \leq 0 \nonumber \\
            &~ 100 y_{i+5}+0.0975 y_{i}^{2}-3.475 y_{i}-9.75 y_{i} y_{i+5} \leq 0 \foralll i \inn [5] \nonumber \\
            &~ \vct{y} \geq (1.000000,1,9,9,9,1,1.000000,1,1,1,50,0.0,1.0,50,50) \nonumber \\
            &~ \vct{y} \leq (8.037732,9,9,9,9,1,4.518866,9,9,9,100,50,50,50,50) \nonumber 
\end{align}
Six of the fifteen variables in Example 9 have matching upper and lower bounds -- these are the six equality constraints alluded to in Table \ref{tab:contemp_summary}.
Our formulation differs from \cite{HSC2014}-8, in that a constraint ``$x_3 x_2 - x_3 \leq 0$'' in the original problem statement was replaced by ``$y_2 - y_3 \leq 0$'' in our problem statement.
This change is necessary because the original problem is actually infeasible.

We approach Example 9 by maximizing our use of partial dualization: the set $X \subset \R^{15}$ includes all bound constraints, all but two of the first nine inequality constraints, as well as the constraint fourth from the end of the problem statement.
The equality constraints implied for variables $y_3, y_4, y_5, y_6, y_{14}, y_{15}$ are not included in the Lagrangian.
A level-(0,1,0) conditional SAGE relaxation then produces a bound $(f,g)_X^\star \geq 156.2196$ in 0.05 seconds.
By running Algorithm 1L with $\epsilon_{\text{ineq}} = 100$, we subsequently obtain the geometric-form solution
\begin{align*}
    \vct{y}^\star_{1:8} &= (8.037732, 9, 9, 9, 9, 1, 1, 1.15686275) \\
    \vct{y}^\star_{9:15} &= (1.21505203, 1.58987319, 50, 3\text{E-}50, 1, 50, 50).
\end{align*}
The solution $\vct{y}^\star$ is feasible up to forward-error 3.6E-14, and attains an objective value of $156.219629$.
Because this objective matches the SAGE bound, we conclude that $\vct{y}^\star$ is optimal up to relative error 2E-7.

\begin{table}[h!]
    \centering
    \begin{tabular}{ccccccc}
    \hline
    source-num. & $(p, q, \ell)$ & bound & $\epsilon_{\text{ineq}}$ & $\epsilon_{\text{eq}}$ & objective & infeasibility \\
    \hline
    \cite{HSC2014}-5 & (0,1,0) & $9171.00$ & 1.00E-08 & 0 & $10122.493$ & 4.00E-13 \\
    \cite{Xu2014}-5 & (2,2,0) & -0.390 & 1.00E-08 & 1 & -0.3888114 & 5.00E-17 \\
    \cite{Xu2014}-7 & (0,1,0) & $9397.8$ & 1.00E-08 & 1 & $10252.790$ & 8.00E-14 \\
    \hline
    \end{tabular}
    \caption{Signomial programs for which we did not certify optimality, but nevertheless recovered best-known solutions by using SAGE relaxations.
    Columns $\epsilon_{\text{ineq}}$ and $\epsilon_{\text{eq}}$ indicate the value of infeasibility tolerances when running Algorithm 1, prior to feeding the output of Algorithm 1 to COBYLA as part of Algorithm 1L.
    The last two columns list the objective function value and constraint violations for the output of Algorithm 1L.
    \cite{Xu2014}-7 reports a solution with smaller objective value, however that solution violates an equality constraint with forward error in excess of 0.11.}
    \label{tab:contemp_unknown}
\end{table}

\subsection{Polynomial optimization problems from the literature}\label{sec:comp:poly_ex}

Here we review results of the reference hierarchies from Section \ref{sec:condsage_polys:refhier}, as applied to twenty-two polynomial optimization problems from the literature.
We begin with six unconstrained and eight box-constrained problems (drawn from \cite{VLSE} and \cite{RN2008} respectively).
There are two important lessons which we highlight with the box-constrained problems.
First, bound constraints should still be included in the Lagrangian, even if they can be completely absorbed into the set ``$X$'' in a conditional SAGE relaxation.
Second, even if the original problem does not feature many sign-symmetric constraints, it is often easy to infer valid sign-symmetric constraints which can improve performance of a conditional SAGE relaxation.
The remaining eight problems discussed in this section have nonconvex objectives, nonconvex inequality constraints, and constraints that the optimization variables are nonnegative \cite{LTY2017}.
Our experience with such problems is that partial dualization plays a crucial role in solving them efficiently, primarily with the simpler constraints $\vct{x} \geq 0$.

Table \ref{tab:uncon_boundcon_summary} gives problem data for the unconstrained and box-constrained problems; three such problems are reproduced here, as our Examples 10 through 12.
\begin{equation}
 \inf\{ f(\vct{x}) \doteq 4 x_1^2 - 2.1 x_1^4 + x_1^6 / 3 + x^{}_1 x^{}_2 - 4 x_2^2 + 4 x_2^4 \,:\, \vct{x} \inn \R^2 \}   \tag{Ex10}
\end{equation}
The polynomial $f$ in Example 10 is known as the six-hump camel function; its minimum $f_{\R^2}^\star \approx -1.0316$ is attained at two points, which differ only by sign.
By using polynomial modulators, a level-(3,0) relaxation returns a bound $-1.03170$ in 0.63 seconds of solver time on Machine $\workstation$.
By instead solving a level-(0,2) relaxation (i.e. modulating the signomial representative of $f - \gamma$) we obtain $-1.031630 \leq f_{\R^2}^\star$ in 0.19 seconds.
Example 10 shows how the two-parameter hierarchy in Section \ref{sec:condsage_polys:refhier} can be of practical importance.

Our next two examples are box-constrained problems from the work of Ray and Nataraj \cite{RN2008}; their problems ``Capresse 4'' and ``Butcher 6'' serve as our Examples 11 and 12.
A consistent trend for these problems is that even when a feasible set $X$ can be incorporated entirely into an $X$-SAGE cone, it is still useful to take products of constraints, and solve a relaxation such as \eqref{eq:poly_sage_hier_def} which includes those constraints in the Lagrangian.
\begin{align}
     \textstyle\inf_{\vct{x} \in \R^4} ~&~ f(\vct{x}) \doteq  -x_1^{}x_{3}^3 + 4 x_{2}^{}x_{3}^{2} x_{4}^{} + 4 x_{1}^{}x_{3}^{} x_{4}^{2} + 2x_{2}^{}x_{4}^{3}  + 4 x_{1}^{}x_{3}^{} \tag{Ex11} \\ &  \qquad \qquad \qquad + 4 x_{3}^2 - 10 x_2 x_4 - 10 x_{4}^{2} + 2  \nonumber \\
        \text{s.t.} & ~ g_{1:4}(\vct{x}) \doteq \vct{x} + (1, 1, 1, 1) / 2 \geq \vct{0} \nonumber \\
                    & ~ g_{5:8}(\vct{x}) \doteq (1, 1, 1, 1) / 2 - \vct{x} \geq \vct{0} \nonumber 
\end{align}
Letting $X = \{ \vct{x} \in \R^4 \,:\, -0.5 \leq x_i \leq 0.5 \}$, one can compute $(f,g)_{X}^{(1,2,0)} = (f,g)_{\R^4}^\star = -3.1176903$, where the equality is verified by recovering a solution with Algorithm \ref{alg:pop_solrec}.
Example 11 is noteworthy because the recovered solution required no local-solver refinement that occurs in Algorithm \ref{alg:pop_solrec}L.
\begin{align}
    \inf_{\vct{x} \in \R^6} ~&~ f(\vct{x}) \doteq x_{6} x_{2}^{2}+x_{5} x_{3}^{2}-x_{1} x_{4}^{2}+x_{4}^{3}+x_{4}^{2}-1 / 3 x_{1}+4 / 3 x_{4} \tag{Ex12} \\
\text{s.t.} &~ g_{1:6}(\vct{x}) \doteq \vct{x} + (1,~ 0.1,~ 0.1,~ 1,~ 0.1,~ 0.1) \geq \vct{0} \nonumber \\
            &~ g_{7:12}(\vct{x}) \doteq (0,~ 0.9,~ 0.5, -0.1, -0.05, -0.03) - \vct{x} \geq \vct{0} \nonumber
\end{align}
We can produce a tight bound for Example 12 with ordinary SAGE certificates: a level-(0,3,0) relaxation returns $-1.4392999 \leq (f,g)^\star$ in 0.67 seconds.
Solution recovery is not so easy.
Unless we move to a computationally expensive level-(0,3,1) ordinary SAGE relaxation, Algorithm 2 fails to return a feasible point.
Instead, we infer valid inequalities to describe a set ``$X$'' for use in a conditional SAGE relaxation:
\begin{align*}
    &|x_1| \leq 1, \quad |x_2| \leq 0.9, \quad |x_3| \leq 0.5, \quad \text{and} \\
    &\quad 0.1 \leq |x_4| \leq 1, \quad 
    0.05 \leq |x_5| \leq 0.1, \quad  0.03 \leq |x_6| \leq 0.1.
\end{align*}
The resulting level (0,3,0) relaxation can be solved in 0.64 seconds.
We recover a feasible solution with Algorithm 2, and obtain a solution matching the SAGE bound after refinement by COBYLA.
Example 12 reinforces a message from signomial optimization: even if an ordinary SAGE relaxation can produce a tight bound, a conditional SAGE relaxation is likely to fare better with solution recovery.
Example 12 also shows how useful sign-symmetric constraints can be inferred from a problem statement, even when those constraints are weaker than those found in the problem.

\FloatBarrier

\begin{table}[h!]
    \centering
    \begin{tabular}{cc|ccc|c}
    \hline
    source & name & $n$ & $d$ & minimum & SAGE solved \\
    \hline
    \cite{VLSE} & Rosenbrock & variable & 4 & 0 & yes \\
    - & 6-hump camel & 2 & 6 & -1.0316 & yes \\
    - & 3-hump camel & 2 & 6 & 0 & yes \\
    - & Beale & 2 & 8 & 0 & no \\
    - & Colville & 4 & 4 & 0 & no \\
    - & Goldstein-Price & 2 & 8 & 3 & no \\
    \hline
    \cite{RN2008} & L.V. 4 & 4 & 4 & -20.8 & yes \\
    - & Cap 4 & 4 & 4 & -3.117690 & yes \\
    - & Hun 5 & 5 & 7 & -1436.515 & no \\
    - & Cyc 5 & 5 & 4 & -3 & yes \\
    - & C.D. 6 & 6 & 2 & -270397.4 & no \\
    - & But 6 & 6 & 3 & -1.4393 & yes \\
    - & Heart 8 & 8 & 4 & -1.367754 & yes \\
    - & Viras 8 & 8 & 2 & -29 & yes \\
    \hline
    \end{tabular}
    \caption{Results for SAGE on unconstrained and box-constrained polynomial minimization problems.
    Column ``$d$'' indicates the degree of the polynomial to be minimized.
    The Rosenbrock example allows for different numbers of variables, though results from \cite{SAGE2} show SAGE is tight for any number of variables.
    The Beale, Colville, and Goldstein-Price polynomials proved very difficult for optimization via SAGE certificates.}
    \label{tab:uncon_boundcon_summary}
\end{table}

\FloatBarrier

Now we turn to problems from \cite{LTY2017}, featuring nonconvex inequality constraints.
One of these problems was introduced in Section \ref{sec:condsage_polys:worked_ex_2} as ``Example 4,'' and all of these problems have a similar structure to that of Example 4. 
Most importantly, problems featured here include nonnegativity constraints $\vct{x} \geq \vct{0}$.
The natural SAGE hierarchy produces tight bounds for all of these problems;
results are summarized in Table \ref{tab:bsos_examples}.

There are a few subtle distinctions between geometric-form signomial programs (SPs), and nonnegative polynomial optimization problems (POPs) such as those considered here.
While a polynomial optimization problem over $\vct{x} \geq \vct{0}$ may include $x_i = 0$ in the feasible set, geometric-form SPs typically cannot allow this (since there is the possibility of dividing by zero, or encountering indeterminate forms).
Thus solution recovery from SAGE relaxations is nominally more challenging for a true nonnegative POP, relative to a geometric-form SP.
Despite this challenge, Algorithm \ref{alg:pop_solrec}L successfully recovers optimal solutions for all of these problems.
See Table \ref{tab:bsos_solrec} for details.

The other important distinction between geometric-form SPs and nonnegative POPs, is that there exist established Sums-of-Squares based methods for dealing with nonnegative POPs.
Thus it is useful to understand the performance of SAGE-based methods in the context of SOS-based methods for polynomial optimization.
Although SAGE relaxations took a very long time to solve problems \texttt{P4{\_}6} and \texttt{P4{\_}8}, the runtimes for problems such as \texttt{P6{\_}8} are remarkable.
The unspecified machine in \cite{LTY2017} took over 1600 and 200 seconds to solve \texttt{P6{\_}8} with BSOS and SOS respectively, while SAGE can solve the same problem in under 4 seconds on a mid-tier laptop from 2013.
It seems to the authors that SAGE provides a compelling option for nonnegative polynomial optimization problems, at least for low levels of the hierarchy (such as $(1,1,0)$, or $(0,q,0)$ with small $q$).

\begin{table}[h!]
    \centering
    \begin{tabular}{cccccc}
    \hline
    name & $k$ & minimum & $(p, q, \ell)$ & $\workstation$ time (s) & $\laptop$ time (s) \\
    \hline
    \texttt{P4{\_}4} & 8 & -0.033538 & (1,1,0) & 0.47 & 0.7 \\
    \texttt{P4{\_}6} & 7 & -0.060693 & (1,1,1) & 289 & 292 \\
    \texttt{P4{\_}8} & 7 & -0.085813 & (2,1,0) & 396 & 460 \\
    \texttt{P6{\_}4} & 8 & -0.576959 & (1,1,0) & 3.45 & 4.1 \\
    \texttt{P6{\_}6} & 8 & -0.412878 & (1,1,0) & 3.04 & 4.37 \\
    \texttt{P6{\_}8} & 8 & -0.409020 & (1,1,0) & 3.25 & 3.83 \\
    \texttt{P8{\_}4} & 8 & -0.436026 & (1,1,0) & 7.18 & 7.25 \\
    \texttt{P8{\_}6} & 8 & -0.412878 & (1,1,0) & 8.67 & 8.21 \\
    \hline
    \end{tabular}
    \caption{Generic polynomial optimization problems, over the nonnegative orthant.
    Problems can be found in both \cite{LTY2017} and \cite{WLT2018}; names ``\texttt{P}$n$\texttt{\_}$d$'' indicate the number of variables $n$ and degree $d$ of the given problem.
    Column $k$ gives the number of inequality constraints, excluding constraints $\vct{x} \geq \vct{0}$, as well as those which trivially follow from $\vct{x} \geq \vct{0}$.
    SAGE solved all problems listed here, at the indicated level of the hierarchy, and with the indicated solver runtimes (in seconds).}
    \label{tab:bsos_examples}
\end{table}

\begin{table}[h!]
    \centering
    \begin{tabular}{c|cc|cc}
    \hline
    & \multicolumn{2}{c|}{ Algorithm \ref{alg:pop_solrec} } & \multicolumn{2}{c}{ Algorithm \ref{alg:pop_solrec}L } \\
    name & $f(\vct{x})$ & $\min g(\vct{x})$ & $f(\vct{x})$ & $\min g(\vct{x})$ \\
    \hline
    \texttt{P4{\_}4} & -0.033386 & 0.00E-00 & -0.033538 & 0.00E-00 \\
    \texttt{P4{\_}6} & -0.057164 & 4.06E-02 & -0.060693 & -2.44E-14 \\
    \texttt{P4{\_}8} & -0.066671 & 1.42E-01 & -0.085813 & -3.46E-14 \\
    \texttt{P6{\_}4} & -0.570848 & 4.04E-08 & -0.576959 & -1.03E-13 \\
    \texttt{P6{\_}6} & -0.412878 & 5.46E-09 & -0.412878 & -1.68E-13 \\
    \texttt{P6{\_}8} & -0.409018 & 1.07E-07 & -0.409020 & -5.82E-14 \\
    \texttt{P8{\_}4} & -0.436024 & 3.27E-08 & -0.436026 & -2.58E-13 \\
    \texttt{P8{\_}6} & -0.412878 & 2.78E-43 & -0.412878 & -2.55E-12 \\
    \hline
    \end{tabular}
    \caption{Comparison of Algorithm \ref{alg:pop_solrec} and Algorithm \ref{alg:pop_solrec}L for solution recovery for eight nonconvex polynomial optimization problems in the literature (ref. \cite{LTY2017,WLT2018}).
    Both algorithms were initialized with solutions to a level-(1,1,0) conditional SAGE relaxation, and Algorithm \ref{alg:pop_solrec}L always recovers an optimal solution.
    It is especially notable that Algorithm \ref{alg:pop_solrec}L recovers optimal solutions for problems \texttt{P4{\_}6} and \texttt{P4{\_}8}, since level-(1,1,0) relaxations do not produce tight bounds for these problems.}
    \label{tab:bsos_solrec}
\end{table}

\subsection{Minimizing random sparse quartics over the sphere}\label{sec:comp:rand_poly}

In this section we describe how SAGE relaxations perform for minimizing sparse quartic forms over the unit sphere; this particular class of test problems is inspired from similar computational experiments by Ahmadi and Majumdar in their work on LP and SOCP-based inner-approximations of the SOS cone \cite{SDSOS}.

Our method for generating these problems is as follows: initialize $f = 0$ as a polynomial in $n$ variables, and proceed to iterate over all tuples ``$t$'' in $[n]^4$.
With probability $n \log n / n^4$, sample a coefficient $c_t$ from the standard normal distribution, and add the term $c_t \vct{x}^{\vct{\alpha}_t}$ to $f$, where $\vct{\alpha}_t \in [4]^n$ has $\alpha_{tj} = |\{ i : t_{i} = j \}|$.
The expected number of terms in $f$ after this procedure is roughly $n \log n$.
Once a polynomial is generated, we solve a level-(0,2,0) conditional SAGE relaxation for $(f,g)^\star_{\R^n}$, where $g(\vct{x}) = 1 - \vct{x}^\intercal\vct{x}$.\footnote{Because $f$ is homogeneous, $\vct{x}^\intercal \vct{x} = 1$ may be relaxed to $\vct{x}^\intercal \vct{x} \leq 1$ without loss of generality}
The set ``$X$'' in the conditional SAGE relaxation is $X = \{ \vct{x} : g(\vct{x}) \geq 0 \}$.
Figure \ref{fig:rand_quartics_plot} and Table \ref{tab:rand_quartics_runtime} report results for 20 problems in 10 variables, 20 problems in 20 variables, 14 problems in 30 variables, and 10 problems in 40 variables.
\begin{figure}[h!]
    \centering
    \includegraphics[width=0.95\textwidth]{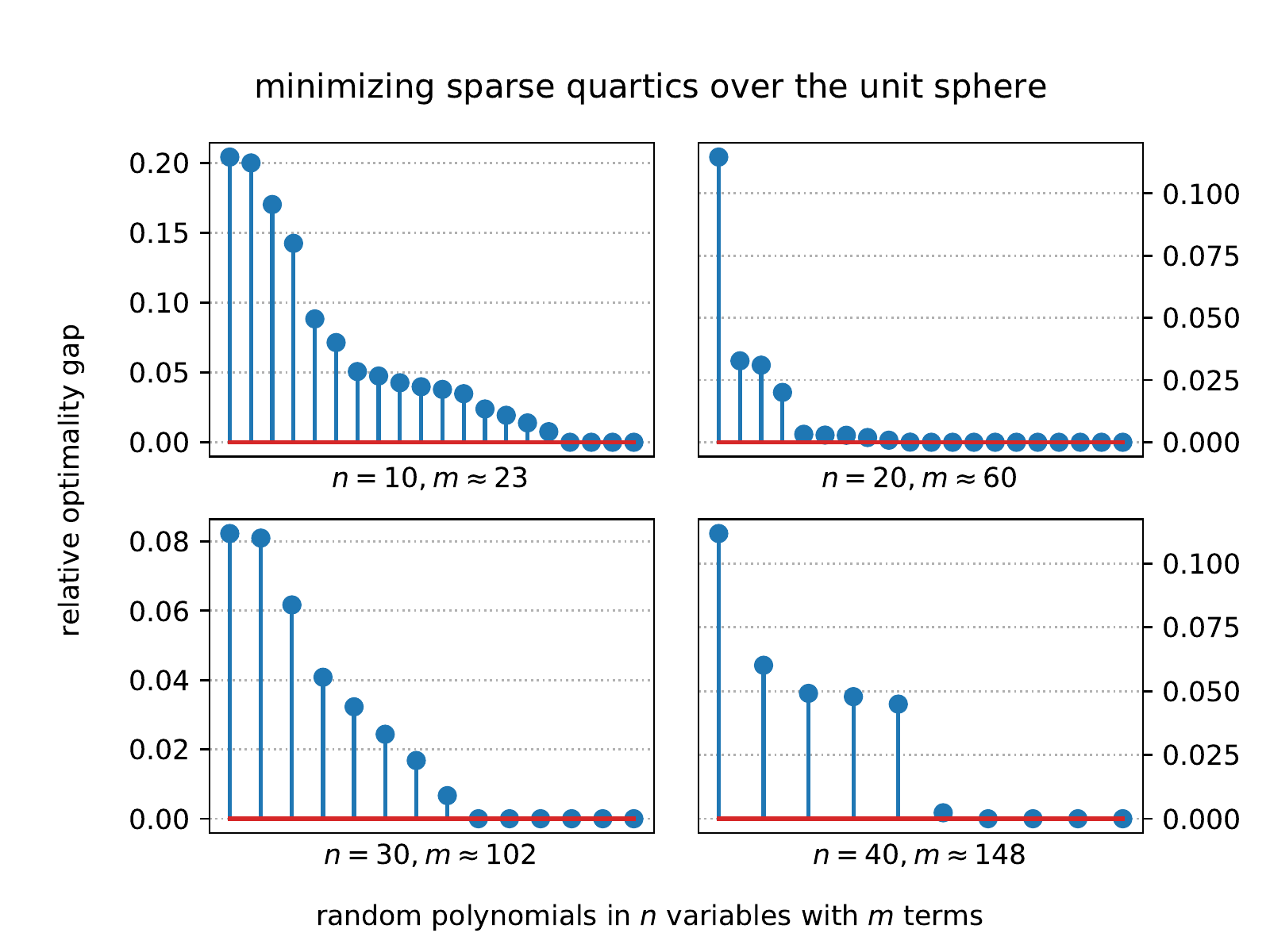}
    \caption{Upper-bounds on the optimality gap $|(f,g)^{(0,2,0)}_X - (f,g)^\star_{\R^n}| / |(f,g)_{\R^n}^\star|$.
    The value $(f,g)^\star_{\R^n}$ in these calculations was replaced by the objective value of a solution produced by Algorithm \ref{alg:pop_solrec}L.
    SAGE solved 4 problems in 10 variables, 10 problems in 20 variables, 6 problems in 30 variables, and 4 problems in 40 variables.}
    \label{fig:rand_quartics_plot}
\end{figure}

\FloatBarrier

\begin{table}[h!]
    \centering
    \begin{tabular}{c|cccc}
    \hline
    solve time (s) & $n=10$ & $n=20$ & $n=30$ & $n=40$ \\
    \hline
    mean & 7.54E-01 & 6.50E-00 & 6.46E+01 & 4.59E+02 \\
    std dev. & 8.74E-02 & 8.54E-01 & 1.38E+01 & 7.20E+01 \\
    \hline
    \end{tabular}
    \caption{Solver runtimes for level-(0,2,0) conditional SAGE relaxations, on Machine $\workstation$.
    Similar runtimes can be expected for Machine $\laptop$ with $n \in \{ 10, 20, 30 \}$. 
    Solve times with Machine $\laptop$ can take much longer for $n \geq 40$, since only a portion of the problem fits in RAM.}
    \label{tab:rand_quartics_runtime}
\end{table}

\FloatBarrier

\section{Discussion and Conclusion}

In this article we introduced and developed notions of conditional SAGE certificates for both signomials and polynomials.
In the signomial case, the underlying theory of the $X$-SAGE cones has deep roots in convex duality, while in the polynomial case, we derived efficient representations of important $X$-SAGE cones by employing ``signomial representatives'' introduced by the authors in earlier work.
Through worked examples and computational experiments, we have demonstrated that subsequent convex relaxations can be used to solve many signomial and polynomial optimization problems from the literature.
The authors believe that conditional SAGE certificates are a fertile area for research in both the theory and practice of constrained optimization; we briefly describe some possible directions here.

From an applications perspective, it would be interesting to see how conditional SAGE certificates can help with branch-and-bound algorithms.
Whether considering signomials in geometric form, signomials in exponential form, or polynomials, bound constraints can easily be incorporated into $X$-SAGE cones.
On the algorithmic front, important work remains to be done on solvers for relative entropy programs, primarily with regards to problems where the optimal solution contains a large number of variables along certain extreme rays of the exponential cone.

Conditional SAGE certificates raise many questions of potential interest to researchers in real algebraic geometry, and optimization-via-nonnegativity-certificates.
Consider for example how the minimax-free hierarchy from Section \ref{sec:condsage_sigs:ref_hier} adopted a particular form for the modulating function: $\Sig(\mtx{\alpha},\vct{1})^{\ell}$.
What benefit might there be to instead using a modulator $\Sig(\mtx{\hat{\alpha}},\vct{1})^\ell$, where $\mtx{\hat{\alpha}}$ was chosen with consideration to $X$? Equally important, how could one efficiently identify good candidates for such $\mtx{\hat{\alpha}}$, given only $\mtx{\alpha}$ and a description of $X$?
And at the most fundamental level, one asks -- for what exponents $\mtx{\alpha}$ and what sets $X$ do $X$-SAGE cones coincide with $X$-nonnegativity cones?
Prior work (c.f. \cite{SAGE2}, and more recently \cite{SONCboundary}) has uncovered meaningful sufficient conditions for this problem when $X = \R^n$.
These sufficient conditions have hitherto been stated in terms of the combinatorial geometry of the exponent vectors $\{ \vct{\alpha}_i\}_{i=1}^m$.
It will be very interesting to see how such results do (or do not) generalize to $X$-SAGE cones for arbitrary $X$.

\printbibliography

\section{Appendix}

\begin{algorithm}[H]
Input: A matrix $\mtx{\alpha} \in \N^{m \times n}$. Vectors $\vct{v} \in \cpolysage{\mtx{\alpha},X}^\dagger$ and $\vct{\hat{v}} \in \csage{\mtx{\alpha},Y}$.
Zero threshold parameter $\epsilon_0 > 0$.
\begin{algorithmic}[1]
\Procedure{VariableMagnitudes}{$\mtx{\alpha}, \vct{v}, \vct{\hat{v}}, \epsilon_0$}
\State $M \gets []$
\For{$j=1,\ldots,m$}
    \If{ $\hat{v}_j = 0$ }
       \State \textbf{Continue}
    \EndIf
    \State Recover $\vct{z}$ in $\R^n$ s.t. $\hat{v}_j \log(\vct{\hat{v}} / \hat{v}_j) \geq [\mtx{\alpha}- \vct{1}\vct{\alpha}_j] \vct{z}~$ and $~(\vct{z},\hat{v}_j) \in \cone Y$.
    \State $\vct{y} \gets \vct{z} / \hat{v}_j$
    \State $M$.append($\exp \vct{y}$)
\EndFor
\If{$(\vct{x}^{\vct{\alpha}_1},\ldots,\vct{x}^{\vct{\alpha}_m}) \neq |\vct{v}|$ for all $\vct{x} \inn M$}
    \State Compute $(\vct{y}, t)$ solving Problem \ref{eq:condsage_polys:solrec:mags:opt_prob}, for given $\epsilon_0$.
    \State $M$.append($\exp \vct{y}$)
\EndIf
\State \textbf{return} $M$.
\EndProcedure
\end{algorithmic}
\caption{magnitude recovery for dual SAGE polynomial relaxations.}
\label{alg:pop_magrec}
\end{algorithm}
As in the signomial case, Algorithm \ref{alg:pop_magrec} always returns a vector $\vct{x} \in X$.
Assuming that $\vct{z}$ from Line 7 are already computed as part of representing $\vct{\hat{v}}$, the complexity of this algorithm is dominated by Line 12.
The runtime of Line 12 is in turn negligible relative to solving a SAGE relaxation to obtain vectors $\vct{v}$ and $\vct{\hat{v}}$.
Infeasibility errors encountered in Line 12 should be handled by jumping to Line 15.

\begin{algorithm}[H]
Input: A matrix $\mtx{\alpha} \in \N^{m \times n}$. A vector $\vct{v}$ in $\R^m$. A Boolean $\mathtt{heuristic}$.
\begin{algorithmic}[1]
\Procedure{VariableSigns}{$\mtx{\alpha}, \vct{v}, \mathtt{heuristic}$}
\State $U \gets \{i \,:\, v_i \neq 0 \text{ and } \vct{\alpha}_i \text{ is not even } \}$
\State $W \gets \{j \,:\, \alpha_{ij} \equiv 1 \mod 2 \text{ for some } i \text{ in } U \}$
\State $Z \gets  \{ \vct{z} \in \{0, 1\}^n \,:\,  \mtx{\alpha}[U,:]\vct{z} \equiv (\vct{v} < 0)[U] \mod 2,~ z_i = 0 \text{ for } i \text{ in } [n] \setminus W \}$
\State $S \gets \{ \}$
\For{$\vct{z}$ in $Z$}
    \State $\vct{s} \gets \vct{1}$
    \For{$j$ in $ \{j \,:\, \alpha_{ij} > 0 \text{ for some } i \text{ in } U \}$}
        \State $s_j \gets -1 \text{ if } z_j = 1, ~ 1 \text{ if } z_j = 0$
    \EndFor
    \State $S \gets S \cup \{\vct{s}\}$
\EndFor
\State If $S = \emptyset$ and $\mathtt{heuristic}$, update $S \leftarrow \{ \text{HueristicSigns}(\mtx{\alpha},\vct{v})\}$.
\State \textbf{return} $S$.
\EndProcedure
\end{algorithmic}
\caption{sign recovery for dual SAGE polynomial relaxations.}
\label{alg:pop_signrec}
\end{algorithm}
Let us describe the ways in which Algorithm \ref{alg:pop_signrec} differs from the discussion in Section \ref{sec:condsage_polys:solrec:signs}.
First- there are changes to the sets $U$ and $W$.
The set $U$ now drops any rows $\vct{\alpha}_i$ from $\mtx{\alpha}$ where $\vct{\alpha}_i$ is even; it is easy to verify that this does not affect the set of solutions to the appropriate linear system.
The set $W$ changes by only considering $j$ where at least one $\alpha_{ij} \equiv 1 \mod 2$. This change is valid because if $\alpha_{ij}$ is even for all $i$, then the sign of variable $x_j$ is irrelevant to the underlying optimization problem, and we make take $x_j \geq 0$ without loss of generality.

Next we speak to the ``hueristic'' sign recovery.
We partly mean to leave this as open-ended, however for completeness we describe the algorithm used in \texttt{sageopt}.
The goal is to find a vector $\vct{s}$ in $\{+1,-1\}$ so that the signs of $\vct{s}^{\mtx{\alpha}} \doteq (\vct{s}^{\vct{\alpha}_1},\ldots,\vct{s}^{\vct{\alpha}_m})$ match the signs of $\vct{v}$ to the greatest extent possible. However, we consider how having $\vct{s}^{\vct{\alpha}_i}$ match the sign of $v_i$ may not be very important if $v_i$ is very small.
Therefore we use a merit function $M(\vct{s}) = \vct{v}^\intercal \vct{s}^{\mtx{\alpha}}$ to evaluate the quality of candidate signs $\vct{s}$.
We apply a greedy algorithm to maximize the merit function $M(\vct{s})$ as follows: initialize $\vct{s} = \vct{1}$, and a set of undecided coordinates $C = \{1,\ldots,n\}$. As long as the set $C$ is nonempty, find an index $i^\star \in C$ so that changing $s_{i^\star} = 1$ to $s_{i^\star} = -1$ maximizes improvement in the merit function. If the improvement is positive, then perform the update $s_{i^\star} \gets -1$.
Regardless of whether or not the improvement is positive, remove $i^\star$ from $C$.
Once $C$ is empty, return $\vct{s}$.

\end{document}